\documentclass{amsart}
\usepackage{fancyhdr}

\usepackage{color}

\usepackage[english]{babel}
\usepackage{eucal}
\usepackage{verbatim}
\usepackage{amsmath}
\usepackage{amssymb}
\usepackage{amsfonts}
\usepackage{mathrsfs}
\usepackage{amsthm}
\usepackage{lscape}
\usepackage{hyperref}
\usepackage[all]{xy}
\usepackage{nth}
\usepackage{caption}
\usepackage{mathabx}
\usepackage{tikz}
\DeclareMathAlphabet{\mathpzc}{OT1}{pzc}{m}{it}

\newcommand{\F}{\boldsymbol{\mathrm{F}}}
\newcommand{\Z}{\boldsymbol{\mathrm{Z}}}

\newcommand{\St}{\boldsymbol{\mathrm{S}}}
\newcommand{\Af}{\boldsymbol{\mathrm{A}}}

\newcommand{\Gl}{\boldsymbol{\mathrm{GL}}}
\newcommand{\M}{\boldsymbol{\mathrm{M}}}
\newcommand{\Ga}{\boldsymbol{\mathrm{G}_a}}
\newcommand{\Gag}{\boldsymbol{\mathrm{G}_a^*}}
\newcommand{\Gar}{\boldsymbol{\mathrm{G}_a}_{\rr}}
\newcommand{\Garr}{\boldsymbol{\mathrm{G}_a^*}_{\rr}}

\newcommand{\rr}{{(r)}}

\newcommand{\tot}{\mathrm{tot}}

\newcommand{\bld}{\boldsymbol}

\newcommand{\im}{{\rm im }}

\newcommand{\Vr}{ V_r }
\newcommand{\Vrr}{ V_r^* }
\newcommand{\HH}{H^{*,*}}

\newcommand{\Ext}{{ \rm Ext}}

\newcommand{\e}{\varepsilon}

\newcommand{\uExt}{{ \rm \underline{Ext}}}

\newcommand{\Hom}{{\rm Hom}}

\newcommand{\Homg}{{\rm Hom}_{\Gr}}

\newcommand{\uHom}{{\rm \underline{Hom}}}

\newcommand{\Gr}{\mathcal{G}\mathcal{R}}
\newcommand{\nil}{{\rm nil}}
\newcommand{\ch}{{\rm char}}

\newcommand{\kf}{\boldsymbol{\rm k}}

\renewcommand{\_}{\rule{1.2ex}{.5pt}\,}

\newtheorem{theorem}{Theorem}[section]
\newtheorem{lemma}[theorem]{Lemma}
\newtheorem{proposition}[theorem]{Proposition}
\newtheorem{corollary}[theorem]{Corollary}

\theoremstyle{definition}
\newtheorem{definition}[theorem]{Definition}
\newtheorem{example}[theorem]{Example}

\newtheorem{rk}[theorem]{Remark}

\title[Graded $1$-parameter subgroups and detection properties]{Graded $1$-parameter subgroups and detection properties}
\author{Camil I. Aponte Rom\'{a}n} 
\email{camili@gmail.com}
\urladdr{www.math.washington.edu/\textasciitilde camili} 
\begin{document}
\maketitle

\tableofcontents
\hbadness=10000
\hfuzz=50pt

 \begin{abstract}

We use tools of representation theory to get a better understanding of the cohomology of graded group schemes. For that, we focus our attention on the case in which the base field is of characteristic $p > 0$. Using as inspiration the work on \cite{SFB1}, \cite{SFB2} and \cite{FP05}, we build the theory of graded $1$-parameter subgroups denoted by $\Vrr(G)$. We  give a natural homomorphism of bigraded $\kf$-algebras $\psi: H^{*, *}(G, \kf) \to \kf[\Vrr(G)],$ where $\kf[\Vrr(G)]$ is the bigraded coordinate ring for $\Vrr(G)$.  We show that $\psi$ is an $F$-monomorphism for a class of graded group schemes. This provides evidence that with the appropriate detection property, a Quillen-type result could exist for graded group schemes.


\end{abstract}

\section{Introduction}

Graded group scheme and graded group varieties are defined, characterized, and studied in \cite{AR1}. The purpose of this paper is to use tools of representation theory to get a better understanding of the cohomology of graded group schemes and graded group varieties. For that, we focus our attention on the case in which the base field is of characteristic $p > 0$. Using the work of \cite{SFB1} and \cite{SFB2} as inspiration, we build the theory of graded $1$-parameter subgroups denoted by $\Vrr(G)$. Ungraded $1$-parameter subgroups play the same role that elementary abelian $p$-subgroups or shifted subgroups play when studying finite groups. An analogue of Quillen's result holds: the cohomology classes of infinitesimal group schemes can be detected, up to nilpotents, by $1$-parameter subgroups. Wilkerson explored this question for graded cocommutative Hopf algebras in \cite{Wi}. In his work, elementary sub-Hopf algebras take the place of elementary abelian $p$-subgroups. He shows that, in this context, Quillen's result no longer holds. More precisely, he shows that for each prime $p$, there exists a graded cocommutative Hopf algebra and a nonnilpotent cohomology class which restricts to zero on every abelian sub-Hopf algebra. We merge both concepts of $1$-parameter subgroups and elementary Hopf algebras and provide evidence that with some modification a result like that of Quillen could exist for graded group schemes. 

As part of the exploration of this idea we give a natural homomorphism of bigraded $\kf$-algebras $\psi: H^{*, *}(G, \kf) \to \kf[\Vrr(G)],$ where $\kf[\Vrr(G)]$ is the bigraded coordinate ring for $\Vrr(G)$. From \cite{SFB1} and \cite{SFB2} we have that for any infinitesimal group scheme of height $\leq r$, the ungraded version of $\psi$ is an $F$-isomorphism. A result of this type allows us to compute the cohomology of a group scheme $G$ (up to nilpotents) in a fairly straightforward way. We can do this by embedding $G$ into some $\Gl_n$ and using the defining equations of this embedding to describe $\kf[\Vr(G)]$ as a quotient of $\kf[\Vr(\Gl_n)]$.

Given the usefulness of such a result, it is a natural question to ask if we get a similar result in our case. In our quest to answer this question we compute some examples of the map $\psi$ and we show that they are in fact $F$-isomorphisms. Then, we define some detection properties of the cohomology of a gr-group scheme and study their relation. Finally, we (partially) answer the question for a class of gr-group schemes that satisfy one of these detection properties. 

\section{Definitions and examples}

We recall some definitions from \cite{AR1}. 
\begin{definition}
Let $\Gr$ be the category of (finitely generated) graded commutative $\kf$-algebras. A
representable functor $G: \Gr \to (groups)$ is called an \textit{affine graded group
scheme}.  We will call them \textit{gr-group schemes} for short. The graded algebra representing $G$ is denoted by $\kf[G]$ and is called the \textit{coordinate algebra} of $G$. We will drop the word affine from now on, as all our gr-schemes will be assumed to be affine. 

\end{definition}

Recall that a graded algebra is graded commutative if, for $a, b$ homogeneous elements in $A$, we have that $ab = (-1)^{|a||b|}ba$. As in the ungraded case, by Yoneda's Lemma there is an equivalence of categories between gr-group schemes and
graded commutative Hopf algebras.

\begin{definition}
We denote $\kf[x_1, \ldots, x_n]^{gr}$ to be the \emph{graded polynomial ring} over $\kf$ in $n$-variable, where $x_ix_j = (-1)^{|x_i||x_j|}x_jx_i$. Note that if $\ch(\kf) \neq 2$, then $x_i^2 = 0$ if $|x_i|$ is odd. 
\end{definition}

\begin{definition}
 We say that a gr-group scheme $G$ is a \textit{finite gr-group scheme}
if $\kf[G]$ is finite dimensional. In that case we can define $\kf G$ as the graded dual of $\kf[G]$; $\kf G$ is a called the \textit{group algebra} for $G$. 
\end{definition}

\begin{definition}
 We say that a gr-group scheme $G$ is a \textit{positive gr-group scheme}
if $\kf[G]$ is positively graded; that is $\kf[G] = \bigoplus_{i \geq 0} (\kf[G])_i$. 
\end{definition}

\begin{definition}
Let $G$ be a gr-group scheme, and let $A = \kf[G]$. If $A_0$ is a local ring and $A$ is positively graded, of finite type (that is, each $A_i$ is finite dimensional) we say that $G$ is a \emph{graded group variety (gr-group variety)}. 
\end{definition}

\begin{example}[Dual Steenrod subalgebra $A(1)$]\label{A1}

Consider $A(1) = \F_2[\xi_1, \xi_2]^{gr}/(\xi_1^4, \xi_2^2)$, where $|\xi_1| = 1$ and $|\xi_2| =
3$. The gr-Hopf algebra structure on $A(1)$
is given by $\Delta(\xi_1) = \xi_1 \otimes 1+ 1 \otimes \xi_1$, and
$\Delta(\xi_2) = \xi_2 \otimes 1 + \xi_1^2 \otimes \xi_1 + 1 \otimes \xi_2$.

As sets $\Homg(A(1), R) =  G_1(R) \times G_3(R)$ where $G_1(R) = \nil^{4}_1(R)$ and $G_2(R) = \nil^2_3(R)$ with product given by $(x,u) \ast (y,v) = (x+y, u+v+x^2y)$ and inverse given by $(x, u)^{-1} = (x, u+x^3)$. We will denote the gr-group scheme represented by $A(1)$ by 
$\St_1$.\end{example}
\section{Graded Frobenius}

In the ungraded setting, the additive group scheme $\Ga$ plays an important role. More precisely its $r$th Frobenius kernel $\Gar$ is of main importance in representation theory. Given a group scheme $G$, the infinitesimal $1$-parameter subgroups of height $\leq r$ in $G$ are defined as group homomorphism from $\Gar \to G$ and denoted by $\Vr(G)$. In the following sections we define the additive gr-group scheme $\Gag$ and its $r$th Frobenius kernel, $\Garr$. To do so, we first give precise definitions of the graded Frobenius map and twist.  We define the graded general linear group, $\Gl_I$. 
As in the ungraded case, $\Gl_I$ proves useful when computing $\Vrr(G)$.

\begin{definition}
Let $A$ be a gr-commutative ring. We define the \textit{Frobenius $p$th power map} $F:A \to A$ such that $F(a) = a^p$. 
\end{definition}

\begin{definition}
Let $A$ be a gr-commutative ring and $M$ a gr-module over $A$. For any $r \geq 0$, the \textit{$r$-th Frobenius twist} of $M$ is the $A$-module $M^\rr = M  \otimes_{F^r} A$. We grade $M^\rr$ as follows: $(M^\rr)_k = \bigoplus_{p^ri+j = k} M_i \otimes_{F^r} A_j$. 
\end{definition}
The degree is well-defined and on $M^\rr$ a homogeneous element $m \otimes 1$ has degree $p^r|m|$. Moreover on $M^\rr$,  $|m \otimes a^{p^r}| = p^r|m| + p^r|a| = p^r(|m|+|a|)$ and also $m \otimes a^{p^r} = am \otimes 1$ which is of degree $p^r(|m|+|a|)$.

\begin{proposition}Let $A$ be a gr-commutative ring and $M$ and $N$ be gr-modules over $A$. For $r \geq 0$, the graded $r$-th Frobenius twist $I^{r}: M \mapsto M^\rr$ is an additive functor from the category of graded $A$ modules to itself. Moreover for any $M,N$ $$(M \otimes_A N)^\rr = M^\rr \otimes_A N^\rr; \,\, (M^\#)^\rr = (M^\rr)^{\#}.$$
\end{proposition}

\begin{proof}
Let $f:M \to N$ be a graded map, then $I^\rr(f) : (m \otimes 1 \mapsto f(m) \otimes 1)$. Now $|m \otimes 1| = p^r|m| \mapsto p^r|f(m)| = p^r|f| + p|m|$. Hence $I^\rr(f)$ is a graded map of degree $p^r|f|$ and it easily follows that $I^\rr$ is an additive functor. The equalities in the statement are true in the ungraded case. For the graded case they follow the same way. We just need to check that the gradings are compatible, which they are by the computations below.  
\begin{align*}
(M \otimes N)^\rr_k &  = \bigoplus_{p^ri +j = k} ( \bigoplus_{l+m = i} M_l \otimes N_m ) \otimes_{F^r} A_j \\
& = \bigoplus_{\substack{p^ri+j = k\\ l+m = i\\ u+v = j}} (M_l \otimes_{F^r} A_u) \otimes (N_m \otimes_{F^r} A_v) ) )\\
(M^\rr \otimes N^\rr)_k &= \bigoplus_{l+m = k} M^\rr_l \otimes N^\rr_m \\
& = \bigoplus_{l+m = k} ( \bigoplus_{p^ri+j = l} M_i \otimes_{F^r} A_j ) \otimes (\bigoplus_{p^ru+v = m} N_u \otimes A_v)\\
& = \bigoplus_{\substack{l+m = k \\ p^ri+j = l \\ p^ru+v = m}} ( M_i \otimes_{F^r} A_j ) \otimes ( N_u \otimes_{F^r} A_v)
\end{align*}

A homogeneous element in $ (M^\#)^\rr$ is of the form $f \otimes_{F^r} 1$, where $f: M \to A$ is a graded map and $|f \otimes_{F^r} 1| = p^r|f|$. This element corresponds to $\widehat{f} \in (M^\rr)^\#$ where $\widehat{f}(m \otimes_{F^r} 1) = f(m) \otimes  1$ with $f : M \to A$ and $|\widehat{f}| = p^r|f|$. 
\end{proof}

\begin{definition}
Given a gr-group scheme $G$ over a gr-commutative ring $A$, since $I^\rr$ is a functor, $A[G^\rr] = A[G] \otimes_{F^r} A$ is a gr-Hopf algebra. 
Then $F^r:A \to A$ ($a \mapsto a^{p^r}$) gives a map $F^r: G \to G^\rr$; we denote by $G_{(r)} = G \times_{G^\rr} \{e\}$ the gr-group scheme which is the kernel of $F^r$. It has coordinate algebra $A[G]/ (\{ x^{p^r} \,\, | \,\,  x \in \ker(\e)\}A[G])$. At the level of gr-Hopf algebras $F^r: A[G^\rr] \to A[G]$ is the map that sends $a \otimes \alpha \mapsto a^{p^r} \otimes \alpha$. We say that $G_\rr$ is the \emph{graded $r$th Frobenius kernel of $G$}. 
\end{definition}

\begin{definition}
A gr-group scheme $G$ is of \emph{height $r$}, if $r$ is the smallest positive integer such that $a^{p^r} = 0$ for every element $a$ in the augmentation ideal of $\kf[G]$, where the augmentation ideal of $\kf[G]$ is the kernel of the counit map $\e: \kf[G] \to \kf$.\end{definition}

\section{Additive graded group scheme}
For $p >2$, let $\Gag$ be the gr-group scheme represented by $\kf[t,s]^{gr}$ where $|t|$ is even and $|s|$ is odd and both $t$ and $s$ are primitive. We call $\Gag$ \textit{additive gr-group scheme}. We would like to point out the difference between the additive gr-group scheme $\Gag$, and the ungraded additive group scheme  $\Ga$. Note that we do not fix the degrees of $t$ nor $s$. They are left as placeholders. 

In the ungraded setting $\kf[\Ga] = \kf[t]$ and  $\Ga(R) = (R, +)$ for any commutative algebra $R$, hence $\Ga$ is the group scheme that describes the additive group structure for any algebra $R$. In the graded case, for any graded algebra $R$, $\Gag(R) = (R_{|t|} \times R_{|s|}, +)$. We do not get the additive group structure of the whole ring $R$, we only get the additive structure of the parts of degree equal to $t$ and $s$. In order to get the full additive structure of $R$ we would need to consider the gr-group scheme represented by $\bigotimes_{i \in Z} \kf[x_i]^{gr}$  where $|x_i| = i$ and $x_i$ is primitive. This gr-Hopf algebra is too big (not noetherian). It turns out that considering $\Gag$ is enough in our setting. In particular, the $r$th Frobenius kernel of $\Gag$, denoted by $\Garr$ and represented by $\kf[t,s]^{gr}/(t^{p^r})$ would be our main protagonist in the development of graded $1$-parameter subgroups and in the further understanding of the cohomology of a gr-group scheme. 

For $p = 2$, our definition of $\Gag$ is slightly different. This is because this case  `behaves' like the ungraded case; our algebras are commutative, not just graded commutative. It is enough to let $\kf[\Gag] = \kf[t]^{gr}$ where $t$ can be even or odd, and in that case, $\kf[\Garr] = \kf[t]^{gr}/(t^{2^r}) = \kf[t]/(t^{2^r})$. 

\section{Graded general linear group}\label{GLI}

\begin{definition}
Let $I = (I_1, \ldots, I_n) \in \Z^n$, then for any graded commutative algebra $A$ we define the gr-group scheme

$$\Gl_I(A) = \{(a_{ij}) \in \Gl_n(A) \,\, | \,\, |a_{ij}|= I_j-I_i \}.$$

The coordinate algebra for $\Gl_I$ is $\kf[\Gl_I] = \kf[x_{ij}, t]^{gr}_{1 \leq i, j \leq n}/(\det(x_{ij})t -1)$, where $\Delta(x_{ij}) = \sum x_{ik} \otimes x_{kj}$, $\e(x_{ij}) = \delta_{ij}$, and $|x_{ij}| = I_j - I_i$.  
The gr-group scheme $\Gl_I$ is called the \textit{graded general linear group indexed by $I$}.
\end{definition}

\begin{definition}
Let $I = (I_1, \ldots, I_n) \in \Z^n$, then for any graded commutative algebra $A$ we define the graded affine scheme

$$\M_I(A) = \{(a_{ij}) \in \M_n(A) \,\, | \,\, |a_{ij}|= I_j-I_i \}.$$

The coordinate algebra for $\M_I$ is $\kf[\M_I] = \kf[x_{ij}]^{gr}_{1 \leq i, j \leq n}$, where $\Delta(x_{ij}) = \sum x_{ik} \otimes x_{kj}$, $\e(x_{ij}) = \delta_{ij}$, and $|x_{ij}| = I_j - I_i$.  
We call it, the graded affine scheme $\M_I$ of \textit{matrices indexed by $I$}.
\end{definition}

\begin{rk}
Note that  $\Gl_I$ is a gr-group scheme as $\Delta$ and $\e$ gives an antipode $S$ (using Cramer's rule). Instead, $\M_I$ is just a graded affine scheme, that is, $\M_I$ is represented by $\kf[\M_I]$, but for a gr-commutative algebra $A$, $\M_I(A) = \Homg(\kf[\M_I], A)$, is not a group.

The quotient map $\kf[\M_I] \to \kf[\Gl_I]$ gives an inclusion $\Gl_I \hookrightarrow \M_I$ and any gr-group scheme mapping to $\M_I$ factors through $\Gl_I$. 
\end{rk}

\begin{proposition}\label{coaction}
Let $G$ be a gr-group scheme. If $M$ is a finite dimensional graded $\kf[G]$-comodule of dimension $n$, then the coaction corresponds to a map $G \to \Gl_I$ for some $I = (I_1, \ldots, I_n)$. 
\end{proposition}
\begin{proof}
Let $M$ be a graded finite dimensional $\kf[G]$ right-comodule with comodule map $\rho: M \to M \otimes \kf[G] $. Let $\{v_1, \ldots, v_n \}$ be a homogeneous basis for $M$.  Let $I = (I_1, \ldots, I_n) = (|v_1|, \ldots, |v_n|)$. Then $\rho(v_j) = \sum v_i \otimes a_{ij} $ where $|a_{ij}| = I_j-I_i$.  Since $v_j = \sum  v_i\e(a_{ij})$ and the $v_i$'s are linearly independent we have that $\e(a_{ij}) = \delta_{ij}$. 

The map $\phi: \kf[\M_I] \to \kf[G]$ given by $x_{ij} \mapsto a_{ij}$ is a gr-bialgebra morphism. To see that, $\phi(\e(x_{ij})) = \phi(\delta_{ij}) = \delta_{ij} = \e(a_{ij}) = \e(\phi(x_{ij})) $. Also since $(\rho \otimes id) \circ \rho = (id \otimes \Delta) \circ \rho$ we get that $\sum_i v_i \otimes \Delta(a_{ij}) = \sum_i v_i \otimes \sum_k a_{ik} \otimes a_{kj}$, hence $\Delta(a_{ij}) = \sum_k a_{ik} \otimes a_{kj}$. Hence $\Delta(\phi(x_{ij})) = \phi(\Delta(x_{ij}))$.  

Then the map $\phi$ yields a gr-algebra map from  $\phi^*: G(R) \to \M_I(R)$ given by $\phi^*(f)(x_{ij}) = f(a_{ij})$ for $f \in G(R)$. Since $G(R)$ is actually a group then $\phi^*$ must factor through $\Gl_I(R)$ as desired. 
\end{proof}

\begin{proposition}\label{finGln}
Let $G$ be a finite gr-group scheme, then there is a closed gr-subgroup embedding $G \hookrightarrow \Gl_I$ for some index $I$. 
\end{proposition}
\begin{proof}

Since $\kf[G]$ is assumed to be finitely dimensional, let $f_1, \ldots, f_n$ be a homogeneous basis for $\kf[G]$ and $I = (|f_1|, \ldots, |f_n|)$. Since $\kf[G]$ is a finite dimensional comodule over itself via the comultiplication map, write $\Delta(f_j) = \sum f_i \otimes g_{ij}$. 

We claim that $(g_{ij}) \in \Gl_I(\kf[G])$. Consider the map $\phi: \kf[\M_I] \to \kf[G]$, where $x_{ij} \mapsto g_{ij}$. We can check that $\phi$ is a gr-bialgebra map. 
Note that since $f_j = \sum \e(f_i)g_{ij}$ and $(\Delta \otimes id) \circ \Delta = (id \otimes \Delta) \circ \Delta$ we get that $\sum_i f_i \otimes \Delta(g_{ij}) = \sum_i f_i \otimes \sum_k g_{ik} \otimes g_{kj}$.  Since the $f_i$'s form a minimal set we get that  $\Delta(g_{ij}) = \sum_k g_{ik} \otimes g_{kj}$. Hence $\Delta(\phi(x_{ij})) = \phi(\Delta(x_{ij}))$. Since $(id \otimes \e) \circ \Delta = id$, we get that $f_j = \sum_i f_i \e(g_{ij})$, which gives that $\e(g_{ij}) = \delta_{ij}$, hence $\e(\phi(x_{ij})) = \phi(\e(x_{ij}))$. Since $(\e \otimes id)\circ \Delta = id$  we get that $f_j = \sum_i \e(f_i) g_{ij}$  and it follows that $\{g_{ij} \}_{ 1 \leq i,j \leq n}$ is a generating set for $\kf[G]$. Therefore $\phi$ is a surjective gr-bialgebra map which yields a map $\phi^*: G(R) \to \M_I(R)$.  Since $G(R)$ is a group then $\phi^*$ must factor through $\Gl_I(R)$ as desired.
\end{proof}

\begin{proposition}(From \cite[3.3]{W})
Let $A$ be a (graded) finitely generated Hopf algebra. Let $V$ be a (graded) comodule for $A$.  Then any (homogeneous) $v \in V$ is contained in a finite dimensional (graded) subcomodule. 
\end{proposition}
\begin{proof}
Let $\{a_{i} \}$ be a (homogeneous) basis for $A$ then $\rho(v) = \sum v_i \otimes a_i$ where all but finitely many $v_i$'s are zero (they are homogeneous). Write $\Delta(a_i) = \sum r_{ijk} a_j \otimes a_k$. 
Then 
$$\sum \rho(v_i) \otimes a_i = (\rho \otimes id)\rho(v) = (id \otimes \Delta)\rho(v) = \sum v_i \otimes r_{ijk}a_j \otimes a_k.$$

We get that $\rho(v_k) = \sum v_i \otimes r_{ijk} a_j$. Hence the space spanned by $v$ and the $v_i$'s is a finite dimensional (graded) subcomodule of $V$. 
\end{proof}

\begin{rk}
We now prove the same results in the case of $G$ any gr-group scheme not necessarily finite.  
\end{rk}

\begin{proposition}\label{glnI}
Let $G$ be a gr-group scheme, then there is a closed gr-subgroup embedding $G \hookrightarrow \Gl_I$. 
\end{proposition}
\begin{proof}

Let $V$ be a graded finite dimensional subcomodule of $\kf[G]$ containing a finite set of generators for $\kf[G]$. Let $\{v_1, \ldots, v_n\}$ be a homogeneous basis for $V$ and $I = ( |v_1|, \ldots, |v_n|)$. Note that $\{v_i \}$ is a generating set for $\kf[G]$. Then $\Delta(v_j) = \sum v_i \otimes a_{ij}$ and $v_j = (\e \otimes id)\Delta(v_j) = \sum \e(v_i) a_{ij}$, therefore $\{a_{ij} \}$ also generate $\kf[G]$. Again consider the map $\kf[\M_I] \to \kf[G]$ which sends $x_{ij} \mapsto a_{ij}$ then the map is surjective since the $a_{ij}$ generate $\kf[G]$ and is a gr-bialgebra map by the same computations as in \ref{finGln} and it gives a closed embedding of $G$ in $\Gl_I$. 
\end{proof}

\begin{example}\label{emSt}
Recall the Dual Steenrod subalgebra $A(1)$ from \ref{A1}, $A(1) = \F_2[\xi_1, \xi_2]^{gr}/(\xi_1^4, \xi_2^2)$, then the generators $\xi_1,\xi_2$ are contained in the finite dimensional comodule $ V = \langle 1, \xi_1,\xi_1^2, \xi_2 \rangle$; a matrix corresponding to the comodule action on $V$ is given by 

$$\phi= \begin{bmatrix}
1 & \xi_1 & \xi_1^2 & \xi_2 \\
0 & 1 & 0 & 0 \\
0 & 0 & 1 & \xi_1 \\
0 & 0 & 0 & 1
\end{bmatrix}.$$

 Now $\det(\phi) = 1$, which is invertible; then the map $\phi^*: \kf[\Gl_I] \to \kf[\St_1]$ which sends $x_{ij} \mapsto \phi_{ij}$ yields the embedding $\phi: \St_1 \hookrightarrow \Gl_I$, where $I = (0, |\xi_1|, |\xi_1^2|, |\xi_2|) = (0,1,2,3)$, as a gr-group scheme $\St_1 \hookrightarrow \Gl_I$. 
\end{example}

\section{Graded $1$-parameter subgroups}

We recall the definition of $1$-parameter subgroups as given in \cite{SFB1}. 

\begin{definition}
Let $G$ be a group scheme over $\kf$ and $ r> 0$ be a positive integer. We define the functor $$\Vr(G): (\textrm{comm $\kf$-alg}) \to (\textrm{sets})$$
by 
$$\Vr(G)(A) = \Hom_{grp/A}(\Ga_{(r)} \otimes A, G \otimes A).$$

That is, all the $A$ group scheme homomorphisms $\Ga_{(r)} \otimes A \to G \otimes A$. 
\end{definition}

By \cite{SFB1} for any commutative algebra $A$, we have the natural identification of 
$$\Vr(\Gl_n)(A) = \{ (a_0, \ldots, a_{r-1}) \,\, | \, \, a_i \in \M_n(A), a_i ^p = 0 = [a_j, a_l] \textrm{ for all } 0 \leq i, j, l < r \}. $$

We modify the definition of the scheme $\Vr(G)$ in the graded case. 

\begin{definition}
Let $G$ be a gr-group scheme, and $r >0$ a positive integer, we define the functor

$$\Vrr(G): \Gr \to (\textrm{sets})$$
by setting

$$\Vrr(G)(A) = \Hom_{\textrm{gr-grp}/A}(\Garr\otimes A, G \otimes A).$$

 We call $\Vrr(G)$ the  \textit{graded $1$-parameter subgroups} for $G$.

\end{definition}

Let $G = \Gl_I$, we see how $\Vrr(\Gl_I)$ compares to the ungraded version defined above. We discuss the $p >2$, the $p =2$ follows similarly and it is easier to compute. Recall that in this case $\Garr$ is represented by $\kf[t,s]^{gr}/(t^{p^r})$.

For a graded commutative algebra $A$, an element $f \in \Vrr(\Gl_I)(A)$  is a map of graded $A$-schemes 
$$f: \Garr \otimes A \to \Gl_I \otimes  A$$ and it corresponds to a $\Garr \otimes A$-module ($A[\Garr]$-comodule) structure on $\Sigma^I A : = \bigoplus_{i = 1}^n (\Sigma^{I_i} A) $ where $I = (I_1, \dots, I_n)$ and $\Sigma^{I_i} A$ corresponds to $A$ shifted by $I_i$. Hence it corresponds to a map 

$$\rho_f: \Sigma^I A \to A[\Garr] \otimes_A \Sigma^I A.$$

For $v \in \Sigma^I A$ we can write $$\rho_f(v) = \sum_{k = 0}^{p^r-1}( t^k \otimes \beta_k(v) + t^ks \otimes \sigma_k(v)),$$

where $\beta_k, \sigma_k \in \M_n(A)$.

From the counit diagram for $\rho_f$ we get that $\beta_0$ is the identity map and from the coassociativity diagram we get the following relations:

for $i+j  \leq p^r-1$, 
\begin{equation}\label{beta} \beta_j \beta_i  = { i+j \choose j } \beta_{i+j} = \beta_i \beta_j,\end{equation} and
 \begin{equation} \label{betasigma}\beta_j \sigma_i  =  \beta_i \sigma_j =  \sigma_i \beta_j = \sigma_j \beta_i = { i+j \choose j } \sigma_{i+j},\end{equation} and for $i+j > p^r-1$ all the above values are zero. 
 We also get that for any $i$ and $j$, \begin{equation} \label{sigma}\sigma_i \sigma_j = 0. \end{equation}
 
\begin{proposition}
Let $\alpha_i = \beta_{p^i}$  for $i = 1, \ldots, r-1$ then any  $\beta_j$ and $\sigma_j$ for $j \neq 0$ can be written in terms of the $\alpha_i$'s and $\sigma_0$.
\end{proposition}
 
 \begin{proof} First note that by formula (\ref{betasigma}) any $\sigma_i  = \beta_i \sigma_0$. Now for some $\beta_j$ not of the form $\alpha_i$, we can write $j = p^l q$ where $p^l$ is maximal with respect to dividing $j$ and $q \neq 1$. Let us assume by induction that $\beta_{p^l(q-1)}$ can be written in terms of the $\alpha_i$'s. Then $j = p^l +p^l(q-1)$ and then ${p^l +p^l(q-1) \choose p^l} \beta_j = \beta_{p^l(q-1)} \beta_{p^m}$. It can be shown that $p  \nmid {p^l +p^l(q-1) \choose p^l}$ therefore $\beta_j$ can be described in terms of the $\alpha_i$'s. 
\end{proof}

\begin{rk}\label{recov}
More precisely, as in \cite[1.2]{SFB1}  we get that for any $j = \sum_{l = 0}^{r-1}j_lp^l$ with $0 \leq j_l < p$, $$\beta_j = \frac{\alpha_0^{j_0}\cdots\alpha_{r-1}^{j_{r-1}}}{(j_0!)\cdots(j_{r-1}!)},$$
and 
$$\sigma_i = \frac{\alpha_0^{j_0}\cdots\alpha_{r-1}^{j_{r-1}}\sigma_0}{(j_0!)\cdots(j_{r-1}!)}.$$
\end{rk}

The proposition above tells us that in order to describe a map $\rho_f$ it is enough to have maps $\alpha_i$  for $i = 0, \ldots r-1$ and $\sigma_0$.  From now on we will relabel $\sigma = \sigma_0$ for simplicity.

\begin{proposition}\label{grtuple}
Let $A$ be a graded commutative algebra and let $I = (I_1, \ldots, I_n)$, then 

$$\Vrr(\Gl_I)(A) =\left\{
	\begin{array}{ll}
		 \{ (\alpha_0, \ldots, \alpha_{r-1}) \in \M_n(A)^{r} \} & \textrm{ if } p = 2, \\
		 \{ (\alpha_0, \ldots, \alpha_{r-1}, \sigma) \in \M_n(A)^{r+1} \} &  \mbox{ if } p > 2.
	\end{array}\right.$$

Where
\begin{itemize}
\item $\alpha_k^p = \sigma^2 =  [ \alpha_k, \alpha_l] = [\alpha_k, \sigma] = 0$ for all  $0 \leq k, l < r$, 
\item $|(\alpha_k)_{ij}| = I_j-I_i-p^k|t|$, and 
\item $|\sigma_{ij}| = I_j-I_i-|s|$.
\end{itemize}
\end{proposition}

\begin{proof}

These relations come formulas (\ref{beta}) and (\ref{betasigma}). 
\end{proof}

\begin{proposition}\label{taylor}
Let $A$ be a graded commutative algebra and let $H  \in \Gl_I(A[\Garr])$ such that $H$ yields a graded comodule structure for $\Sigma^I A$. Then

$$H= \sum_{k = 0}^{p^r-1}\beta_k t^k+\sigma_k t^ks$$
\end{proposition}

\begin{proof}
A homogeneous basis for $\Sigma^I A$ over $A$ must be given by $e_1 = 1_{\Sigma^{I_1} A}, e_2 = 1_{\Sigma^{I_2} A}, \ldots e_n = 1_{\Sigma^{I_n} A}$ where then $|e_i| = I_i$ then $\rho_f(e_j) = \sum g_{ij} \otimes e_i$, where $(g_{ij}) \in \Gl_I(A[\Garr])$.

Given $v \in \Sigma^I A$ we can write $v = a_1 e_1+ \cdots + a_n e_n$ and 

$$\rho_f(v) = \sum_{k  = 0}^{p^r-1} t^k \otimes \beta_k(v) + t^ks \otimes \sigma_k(v).$$  

Note that if instead we write the coaction in terms of the basis elements we get that 

$$\rho_f(e_j) = \sum_{i = 1}^{n} g_{ij} \otimes e_i.$$

Comparing these expressions we obtain 

\begin{eqnarray*}
\sum_{i = 1}^n g_{ij} \otimes e_i & = &\sum_{k = 0}^{p^r-1} t^k \otimes \beta_k(e_j) + t^ks \otimes \sigma_k(e_j) \\
& = &\sum_{k = 0}^{p^r-1} t^k \otimes \sum_{i  = 1}^n (\beta_k)_{ij} e_i + t^ks \otimes \sum_{i = 1}^n (\sigma_k)_{ij} e_i \\
& =& \sum_{i = 1}^n \sum_{k = 0}^{p^r-1} [(\beta_k)_{ij} t^k + (\sigma_k)_{ij} t^ks ]\otimes e_i.\\
\end{eqnarray*}

We have this Taylor polynomial type relation:

$$g_{ij} = \sum_{k= 0}^{p^r-1}(\beta_k)_{ij} t^k + (\sigma_k)_{ij} t^ks.$$
\end{proof}

\begin{rk}
For $p >2$, given such a map $f: \Garr \otimes A \to \Gl_I \otimes A$ as above, we get an $r+1$-tuple $(\underline{\alpha}, \sigma) = (\alpha_0, \ldots, \alpha_{r-1}, \sigma) $.

Conversely, given $(\underline{\alpha}, \sigma)$, we can construct the map  $f: \Garr \otimes A \to \Gl_I \otimes A$ as $\exp(\underline{\alpha}, \sigma)$, defined the following way. Let $R$ be a graded $A$-algebra, we define $\exp(\alpha)$ for any $p$-nilpotent $\alpha \in \M_n(R)$ as 
$$\exp(\alpha) = I + \alpha + \frac{\alpha^2}{2}+ \cdots + \frac{\alpha^{p-1}}{(p-1)!} \in \Gl_I(R).$$

There is a correspondence between $u \in R_{|t|}$ ($p^r$-nilpotent),  $v \in R_{|s|}$, and $g \in (\Garr \otimes A)(R)$ given by $g(t) = u$ and $g(s) = v$.  We define $\exp(\underline{\alpha}, \sigma)$ as follows: for any  $u \in R_{|t|}$ and $v \in R_{|s|}$,

$$\exp(\underline{\alpha}, \sigma)(u,v) = \exp(u\alpha_0)\exp(u^p\alpha_1) \cdots \exp(u^{p^{r-1}}\alpha_{r-1}) \exp(v \sigma)\in \Gl_I(R),$$

and from the relations in \ref{recov} we can check that $f = \exp(\underline{\alpha}, \sigma)$. 

Beware that even though $\alpha_i$ and $\sigma$ commute, $(\alpha_i t)(\sigma s)$ may not be equal to $(\sigma s)(\alpha_i t)$ since $s$ is of odd degree, so we need to be careful when computing $\exp(\underline{\alpha}, \sigma)(u,v)$.

 \end{rk}

\begin{proposition}\label{coordVr}
Given a gr-group scheme $G$, $\Vrr(G)$ is a graded affine scheme, that is, $\Vrr(G)$ is a representable functor from graded commutative algebras to sets. Moreover $\Vrr$ is a functor from gr-group schemes to graded affine schemes. 
\end{proposition}

\begin{proof}
By. \ref{grtuple} $\Vrr(\Gl_I)$ has as a coordinate ring;
$$\frac{\kf[X_{ij}^l]^{gr}_{{0 \leq l \leq r-1,  1 \leq i,j \leq n}}}{((X^l)^p,[X^l, X^k])},$$ for $p =2$ and, 
$$\frac{\kf[X_{ij}^l, Y_{ij}]^{gr}_{{0 \leq l \leq r-1,  1 \leq i,j \leq n}}}{((X^l)^p, (Y)^2,[X^l, X^k], [X^l, Y])},$$
for $p >2$.

Where $X^l$ and $Y$ are the $n \times n$ matrices with $ij$th entry being the variable $X^l_{ij}$ and $Y_{ij}$ respectively. 
These are graded by $|X_{ij}^l| = (I_j-I_i-p^l|t|)$ and $|Y_{ij}| = (I_j-I_i-|s|)$. 

Let $G$ be a gr-group scheme; by \ref{glnI} there exists an embedding  $\phi: G \hookrightarrow \Gl_I$ for some $I$. 
Given a graded algebra $A$, elements in $\Vrr(\Gl_I)(A)$ corresponds to an $r+1$-tuple $(\underline{\alpha}, \sigma) \in \M_n(A)^{r+1}$ (satisfying the conditions in \ref{grtuple}) via $\exp(\underline{\alpha}, \sigma)(t,s)$. 
Then given $g \in \Vrr(G)(A)$ there exists an $(\underline{\alpha}, \sigma) \in \Vrr(\Gl_I)$ such that the following diagram commutes.

$$\xymatrix{\Garr \otimes A \ar[rr]^{\exp(\underline{\alpha}, \sigma)(t,s)} \ar[drr]_g && \Gl_I \otimes A\\
 && G \otimes A \ar@{^{(}->}[u]_{\phi \otimes A}}$$
 
The above diagram says that any element in $\Vrr(G)$ can be describe by the defining equations of the embedding on an exponential map to $\Gl_I \otimes A$. More precisely, the embedding $\phi: G \hookrightarrow \Gl_I$ corresponds to a surjective map $\phi^*: \kf[\Gl_I] \to \kf[G]$ and has defining equations $F_1, \ldots, F_m$ describing the kernel of $\phi^*$.  Then the coordinate algebra of $\kf[\Vrr(G)]$ is precisely the quotient of $\kf[\Vrr(\Gl_I)]$ defined by $F_i(\exp(X_{ij}^l, Y)(t,s))= 0$. 
 
Note that since the coordinate ring for $\Vrr(G)$ uniquely describes $\Vrr(G)$ as a graded affine scheme then as a graded algebra $\kf[\Vrr(G)]$ is independent of the embedding. 

We now show that $\Vrr(\_)$ is a functor. Let $\varphi: G \to H$ be gr-group scheme homomorphism, $A$ a gr-commutative algebra, and $\Vrr(\varphi): \Vrr(G) \to \Vrr(H)$ is given by the composition $\xymatrix{\Garr\otimes A \ar[r]^g  & G \otimes A \ar[r]^{\varphi \otimes A} & H \otimes A}$ hence $\Vrr(\varphi)(g) = (\varphi \otimes A)\circ g \in \Vrr(H)$. It is then clear that $\Vrr$ is a covariant functor from gr-group schemes to affine graded schemes.  
\end{proof}

\begin{rk}[Bigraded rings; notation for bidegree and total degree]
There is no standard notation for the bidegree and the total degree of a bigraded ring. We describe the notation we will use. Let $R$ be a bigraded ring. We will use $\|x\| = (i,j)$ to denote the bidegree of a bihomogeneous element $x \in R$. We denote $\|x\|_1 = i$ and $\|x\|_2 = j$, if we wish to refer to the first or second degree. The bigraded ring $R$ can be made into a graded ring via the first or second degree and also by the total degree, we denote the total degree of $x$ by $\tot(x) = i+j$.
\end{rk}

We can make $\kf[\Vrr(G)]$ into a bigraded ring with some external degrees. By convention these degrees do not introduce any Kozsul sign convention. 
For instance, when we write $$\kf[\Vrr(\Gl_I)] = \frac{\kf[X_{ij}^l, Y_{ij}]^{gr}_{{0 \leq l \leq r-1,  1 \leq i,j \leq n}}}{((X^l)^p, (Y)^2,[X^l, X^k], [X^l, Y])}$$ the $gr$ refers to $\kf[\Vrr(\Gl_I)]$ as a graded polynomial ring with respect to the internal degree that was computed in  \ref{coordVr}.

We bigrade $\kf[\Vrr(G)]$ the following way 
\begin{eqnarray*}
\|X_{ij}^l\| &= &\left\{
	\begin{array}{ll}
		 (p^l, I_j-I_i-|t|p^l) & \textrm{ for } p = 2, \\
		(2p^l, I_j-I_i-|t|p^l) &  \mbox{ for } p > 2, \mbox{ and}
	\end{array}\right.\\
	\|Y_{ij}\|& =& (p^r, I_j-I_i-|s|), \mbox{ for } p>2.
	\end{eqnarray*}



\begin{proposition}(From \cite[2.2]{Wi})\label{HGar}
For $\Garr$ the graded cohomology is 

$$\HH(\Garr, \kf) = \kf[\lambda_1, \ldots, \lambda_{r}]^{gr},$$
 where $\|\lambda_i\| = (1, |t|p^{i-1})$ for $p = 2$, and 
 $$\HH(\Garr, \kf) = \kf[x_1, \ldots, x_r, y, \lambda_1, \ldots, \lambda_r]^{gr},$$  
 where $\|x_i\| = (2, |t|p^i), \|y\| = (1, |s|)$, and $\|\lambda_i\| = (1, |t|p^{i-1})$ for $p > 2$. 
 
The cohomology is bigraded with the first degree corresponding to the cohomological degree and the second degree correspond to the internal degree of $\Garr$. As an algebra $\HH(\Garr, \kf)$ is  gr-commutative with respect to the total degree (c.f. \ref{grcomcoh}), that is, the sum of the cohomological and internal degree. 
 \end{proposition}
 
%
%


 \begin{proposition}\label{actions}
The external grading on $\Vrr(G)$ corresponds to an action of $\Af^1$ on $\Vrr(G)$. 
\end{proposition}
\begin{proof}
Since we can embed $G$ into some $\Gl_I$, it is enough to check this in the case of $\Vrr(\Gl_I)$. In that case given $\gamma \in \Af^1$ we have that 

$$\Vrr(\Gl_I) \times \Af^1 \to \Vrr(\Gl_I)$$ where

$$\langle \underline{\alpha},  \gamma \rangle \mapsto
 ( \gamma\alpha_0, \gamma^{p}\alpha_1, \ldots, \gamma^{p^{r-1}}\alpha_{r-1} )$$
 for  $p = 2$, and 
  $$\langle (\underline{\alpha}, \sigma)),  \gamma \rangle \mapsto ( \gamma^2\alpha_0, \gamma^{2p}\alpha_1, \ldots, \gamma^{2p^{r-1}}\alpha_{r-1}, \gamma^{p^r} \sigma)$$
  for $p >2$. 
%
	
\end{proof}

\begin{rk}
The action of $\Af^1$ above is compatible with the action of $\Af^1$ on $\Garr$ where $\langle t, \gamma \rangle \in \Garr \times \Af^1$ maps to $\gamma t$ for $p = 2$ and $\langle (t,s), \gamma \rangle \in \Garr \times \Af^1$ maps to $(\gamma^2 t, \gamma^{p^r} s)$ for $p >2$.

This action does the following on $\HH(\Garr, \kf)$:

 $$\begin{array}{ll}
 \gamma^*(\lambda_i) = \gamma^{2^{i-1}} \lambda_i & \mbox{ for } p = 2,\\
   \gamma^*(x_i) = \gamma^{2p^i}x_i, \gamma^*(y) = \gamma^{p^r} y, \mbox{ and } \gamma^*(\lambda_i) = \gamma^{2p^{i-1}} \lambda_i &  \mbox{ for } p  > 2.
	\end{array}$$   
\end{rk}

\subsection{Computing $\kf[\Vrr(G)]$ for some gr-group schemes}

To compute $\kf[\Vrr(G)]$ we embed our gr-group schemes into some $\Gl_I$ as described in \ref{glnI}. The defining equations of the embedding $\phi: G \hookrightarrow \Gl_I$ corresponds to a surjective map $\phi^*: \kf[\Gl_I] \to \kf[G]$ and has defining equations $F_1, \ldots, F_m$ describing the kernel of $\phi^*$.  Then the coordinate algebra of $\kf[\Vrr(G)]$ is precisely the quotient of $\kf[\Vrr(\Gl_I)]$ defined by $F_i(\exp(X_{ij}^l, Y)(t,s))= 0$ for $p >2$ and by $F_i(\exp(X_{ij}^l)(t))= 0$ for $p =2$.

\begin{example}\label{vrst}
Recall the gr-group scheme $\St_1$ from \ref{A1}, $\St_1$ is represented by $A(1) = \F_2[\xi_1, \xi_2]^{gr}/(\xi_1^4, \xi_2^2)$. From \ref{emSt} the embedding $\phi: \St_1 \to \Gl_I$

$$\phi= \begin{bmatrix}
1 & \xi_1 & \xi_1^2 & \xi_2 \\
0 & 1 & 0 & 0 \\
0 & 0 & 1 & \xi_1 \\
0 & 0 & 0 & 1
\end{bmatrix},$$
where $I = (0, |\xi_1|, |\xi_1^2|, |\xi_2|) = (0,1,2,3)$. 

The gr-group scheme $\St_1$ has height $2$, so we compute $\F_2[V_2^*(\St_1)]$. We have that  $$\F_2[V^*_2(\Gl_I)] = \frac{\kf[X_{ij}^0, X_{ij}^1]^{gr}_{{1 \leq i,j \leq 4}}}{((X^0)^2, (X^1)^2, [X^0, X^1])}.$$

Let  $Z = \exp(X^0, X^1)(t) = I + X^0t +X^1t^2 + X^0X^1t^3$ be the $4 \times 4$ matrix with entries in $\F_2[V_2^*(\Gl_I)]$. From $\phi$ we get that the defining equations are
\begin{itemize}
\item  $Z_{ii} = 1,$
\item $Z_{21} = Z_{23} = Z_{24} = Z_{31} = Z_{32} = Z_{41} = Z_{42} = Z_{43} = 0, $
\item $Z_{12} = Z_{34},$
\item $Z_{12}^2 = Z_{13},$ and
\item $Z_{12}^4 = 0, Z_{14}^2 = 0$.
\end{itemize}

We get some relations on the entries of $X^0$ and $X^1$. To begin with, the only possibly nonzero entries are $X^l_{12}, X^l_{13}, X^l_{14}, X^l_{34}$ for $l = 0, 1$. Also, from the definition of $\F_2[V_2^*(\Gl_I)]$, we also have that $(X^0)^2 = (X^1)^2 = 0$ and $[X^0, X_1] = 0$. Putting all of these together we get that,

\begin{itemize}
\item $X^0_{12} = X^0_{34}, X^1_{12} = X^1_{34},$
\item $(X^0_{12})^2 = X^1_{13}, X^0_{13} = 0,$
\item $(X_{14}^0)^2 = 0$, and 
\item $(X^1_{13})^2 = 0$, $X^1_{13}X^1_{12} = 0$.
\end{itemize}
Therefore $$\F_2[V_2^*(\St_1)] = \frac{\F_2[X^0_{14}, X_{12}^1, X^1_{13}, X^1_{14}]^{gr}}{((X^0_{14})^2, (X_{13}^1)^2, (X_{13}^1X_{12}^1))},$$
which is $F$-isomorphic to $\F_2[X_{12}^1,X^1_{14}]^{gr}$ where $\|X^1_{12}\| = (2, 1-2|t|)$ and $\|X^1_{14}\| = (2, 3-2|t|)$.

\end{example}

\begin{example}(Wilkerson's counterexample \cite[6.3]{Wi})\label{wilk1}

Consider the quotient of the dual of the Steenrod algebra given by $$\F_p[W_1] = \frac{\F_p[\xi_1, \xi_2, \xi_3]^{gr}}{(\xi_1^p, \xi_2^{p^2}, \xi_3^p)}.$$ We can embed $\phi: W_1 \to \Gl_I$ where $$\phi = \begin{bmatrix}
1 & \xi_1 & \xi_2 & \xi_2^p & \xi_3 \\
0 & 1 & 0 & 0 & 0 \\
0 & 0 & 1 & 0 & 0 \\
0 & 0 & 0 & 1 & \xi_1 \\
0 & 0 & 0 & 0 & 1
\end{bmatrix}$$

We compute $\F_p[\Vrr(W_1)]$ for $r = 2$ which is the height of $\F_p[W_1]$. 

\begin{itemize}
\item For $ p = 2$, let $I  = (0,1, 3, 6, 7)$ and $$\F_2[V_2^*(\Gl_I)] = \frac{\F_2[X^0_{ij}, X^1_{ij}]^{gr}_{1 \leq i,j \leq 5}}{((X^0)^2, (X^1)^2, [X^0, X^1])}.$$ 

Then $\exp(X^0, X^1)(t) = I + X^0t +X^1t^2 + X^0X^1t^3$. The defining equations for the embedding give that $X^0_{ij} = X^1_{ij} = 0$ for $ij \neq 12, 13, 14, 15, 45$. This also implies that $(X^0X^1)_{ij} = 0$ except for $ij = 15$ and in that case $(X^0X^1)_{15} = X^0_{14}X^1_{45} = X^1_{14}X^0_{45}$. Among other relations, we get that,  $\F_2[V_2^*(W_1)]$ is $F$-isomorphic to
$$\frac{\F_2[X^0_{13}, X^1_{12}, X^1_{13}, X^1_{15}]^{gr}}{((X^0_{13})^2X^1_{12})},$$
where

\begin{itemize}
\item $\|X^0_{13}\| = (1, 3-|t|)$, 
\item $\|X^1_{12}\| = (2, 1-2|t|)$,
\item $\|X^1_{13}\| = (2, 3-2|t|)$, and 
\item $\|X^1_{15}\| = (2, 7-2|t|)$. 
\end{itemize}

\item For $p > 2$ we have that $I = (0, 2(p-1), 2(p^2-1), 2p(p^2-1), 2(p^3-1))$, and
$$\F_p[V_2^*(\Gl_I)] =  \frac{\F_p[X^0_{ij}, X^1_{ij}, Y_{ij}]^{gr}_{1 \leq i,j \leq 5}}{((X^0)^p, (X^1)^p, (Y_{ij})^2, [X^0, X^1], [X^0, Y], [X^1, Y])}.$$ 

Then doing computations as in the $p = 2$ case, we find that $\F_p[V_2^*(W_1)]$ is $F$-isomorphic to 
$$\F_p[X^0_{13}, X^1_{12}, X^1_{13}, X^1_{15}]^{gr},$$

where
\begin{itemize}
\item $\|X^0_{13}\| = (2, 2(p^2-1)-|t|)$,
\item $\|X^1_{12}\| = (2p, 2(p-1)-p|t|)$, 
\item $\|X^1_{13}\| = (2p, 2(p^2-1)-p|t|)$, 
\item $\|X^1_{15}\| = (2p, 2(p^3-1)-p|t|)$, 
\end{itemize}

\end{itemize}
\end{example}

\begin{example}(Wilkerson's counterexample \cite[6.5]{Wi})\label{wilk2}
Let $$\F_2[W_2] = \frac{\F_2[x_i]_{1 \leq i \leq 5}^{gr}}{(x_i^2)}, $$ and $x_i$ is primitive for $i <5$ and $\Delta(x_5) = x_5 \otimes 1 + x_1 \otimes x_4 + x_ 2\otimes x_3 + 1 \otimes x_5$. 
We can embed $\phi: W_2 \to \Gl_I$ where $$\phi = \begin{bmatrix}
1 & x_1 & x_2 & x_3 & x_4 & x_5 \\
0 & 1 & 0 & 0 & 0 & x_4 \\
0 & 0 & 1 & 0 & 0 & x_3 \\
0 & 0 & 0 & 1 & 0 & 0\\
0 & 0 & 0 & 0 & 1 & 0 \\
0 & 0 & 0 & 0 & 0 & 1
\end{bmatrix},$$

and $I = (0, 1, 2, 3, 4, 5)$. 

The height of $\F_2[W_2]$ is $r =1$, then $$\F_2[V_1^*(\Gl_I)] = \frac{\F_2[X_{ij}]^{gr}_{1 \leq i,j \leq 6}}{(X^2)},$$

where $\|X_{ij}\| = (1, j-i-|t|)$.

By the defining equation of the embedding we get that 

$$\F_2[V_1^*(W_2)] = \frac{\F_2[X_{12}, X_{13}, X_{14}, X_{15}, X_{16}]^{gr}}{(X_{12}X_{15}+X_{13}X_{14})}.$$

\end{example}

\begin{example}\label{Stodd}
Consider the quotient of the dual of the Steenrod algebra given by $$\F_p[G] = \frac{\F_p[\xi_1, \tau_0, \tau_1]^{gr}}{(\xi_1^p)},$$ for $p >2$ where $|\xi_1| = 2(p-1)$, $|\tau_0| = 1$ and $|\tau_1| = 2p-1$, $\xi_1$ and $\tau_0$ are primitive and $\Delta(\tau_1) = \tau_1 \otimes 1 + \xi_1 \otimes \tau_0 + 1 \otimes \tau_1$.

We can embed $\phi: G \hookrightarrow \Gl_I$ where $I = (0, 2(p-1), 1, 2p-1)$ and 

$$\phi = \begin{bmatrix}
1 & \xi_1 & \tau_0 & \tau_1 \\
0 & 1 & 0 & \tau_0 \\
0 & 0 & 1 & 0\\
0 & 0 & 0 & 1 
\end{bmatrix}.$$

Then the height is $r =1$ and $$\F_p[V_1^*(\Gl_I)]  = \frac{\F_p[X_{ij}, Y_{ij}]^{gr}_{1 \leq i,j \leq 4}}{( X^p, Y^2, [X,Y] )}.$$

For a defining equation $F_i$ for $\phi$, the equation $F_i(\exp(X,Y)(t,s)) = 0$ gives us that 

\begin{itemize}
\item $X_{ij} = Y_{ij} = 0$ for $ij \neq 12, 13, 14, 24$, 
\item $(X^2)_{ij} = 0$ except for $(X^2)_{14} = X_{12}X_{24}$, and
\item $(XY)_{ij} = 0$ except $(XY)_{14}  = X_{12}Y_{24} = Y_{12}X_{24}$, 
\item $X^3 =0$ and $XY^2 = 0$,
\item $(X_{13})^2 = (X_{14})^2 = 0$, 
\end{itemize}

among other relations. Then $\F_p[V_1^*(G)]$ is $F$-isomorphic to 
$$\frac{\F_p[X_{12}, Y_{12}, Y_{13}, Y_{14}]^{gr}}{(Y_{12}Y_{13}, X_{12}Y_{13}-Y_{12}X_{13})},$$
with bidegrees 
\begin{itemize}
\item $\|X_{12}\| = (2, 2(p-1)-|t|)$,
\item $\|Y_{12}\| = (p, 2(p-1)-|s|)$, 
\item $\|Y_{13}\| = (p, 1-|s|)$, and
\item $\|Y_{14}\| = (p, 2p-1-|s|)$. 
\end{itemize}
\end{example}

\section{The algebra map $\psi: \HH(G, \kf) \to \kf[\Vrr(G)]$}

For a gr-group scheme $G$ we give an algebra map from $\HH(G, \kf)$ to $\kf[\Vrr(G)]$. For reference we state the results in the ungraded case. 

\begin{theorem}(\cite[1.14]{SFB1})
For any affine group scheme $G$, there is a natural homomorphism of graded commutative algebras
$\psi: H^{ev}(G,\kf) \to \kf[V_r(G)]$ which multiplies degrees by $p^r/2$.\end{theorem}

\begin{theorem}(\cite[1.14.1]{SFB1})
Assume that $p = 2$. Then for any affine group scheme $G$,
there is a natural homomorphism of graded commutative algebras $\psi:H^*(G, \kf) \to \kf[V_r(G)]$ which multiplies degrees by $2^{r-1}$.
\end{theorem}

For a gr-group variety we get the following graded version of $\psi$.

\begin{theorem}\label{HG-KV}
Let $G$ be a gr-group scheme. There is an algebra map
  $$\psi: \HH(G, \kf) \to \kf[\Vrr(G)].$$

If $z \in H^{n, m}(G, \kf)$, then 
\begin{itemize}
\item for $p >2$, $\psi(z)$  is a sum of bihomogeneous pieces of bidegree $(np^r, m-|t|p^rl-k|s|)$ where $2l+k = n$; 
\item for $p = 2$, $\psi(z)$ is bihomogeneous of bidegree $(n2^{r-1}, m-n|t|2^{r-1})$. 
\end{itemize}\end{theorem}

\begin{proof}
Consider the identity map $1 \in \Vrr(G)(\kf[\Vrr(G)])$; this map corresponds to some $u: \Garr \times \Vrr(G) \to G \times \Vrr(G)$. 

For any graded algebra $A$, an element in $ (\Garr \times \Vrr(G))(A)$ corresponds to $\langle (a,b),f\rangle $, where $(a, b)$ is the map in $\Garr(A)$ such that $t \mapsto a$ and $s \mapsto b$ and $f: \Garr\otimes A \to G \otimes A$.

The identity map $1 \in \Vrr(G)(\kf[\Vrr(G)])$ corresponds to $u: \Garr \times  \Vrr(G) \to (G \times \Vrr(G))$ where for $\langle (a,b), f\rangle \in (\Garr \times  \Vrr(G))(A)$,  $u(\langle (a,b), f \rangle) = \langle f(a,b), f\rangle.$ 

To see this we focus on the case of $\Gl_I$.  It can be checked that $$1 \in \Vrr(\Gl_I)(\kf[\Vrr(\Gl_I)])$$ corresponds to $$u = \exp(\underline{X}, Y)(t,s)$$ where $\underline{X}  = (X^0, \ldots, X^{r-1})$. Note also that $f$ is given by an $r+1$-tuple $(\underline{\alpha}, \sigma)$ such that $\exp(\underline{\alpha}, \sigma)(t,s) = f$. Then $u(\langle (a,b), f \rangle) = \langle \exp(\underline{\alpha}, \sigma)(a,b), \exp(\underline{\alpha}, \sigma)(t,s)\rangle$. This can be unraveled by saying that $u$ sends $\langle (a,b), f \rangle \mapsto \langle f(a,b), f \rangle \in (\Gl_I \times \Vrr(\Gl_I))(A)$. 

For a gr-group scheme $G$ and embedding $\phi: G \hookrightarrow \Gl_I$ the following diagram commutes: 

$$\xymatrix{\Garr \times \Vrr(G) \ar@{^{(}->}[r] \ar[d]_u & \Garr \times \Vrr(\Gl_I) \ar[d]^{u}  \\
G \times \Vrr(G) \ar@{^{(}->}[r] & \Gl_I \times \Vrr(\Gl_I),}$$

which gives that $u$ is as described above. 

We now define $\psi$; its definition will depend on whether $p =2$ or $p >2$, recall $\HH(\Garr, \kf)$ from \ref{HGar}. 

\begin{itemize}
\item For $p = 2$ and $z \in H^{n,m}(G, \kf)$, $$u^*(z) = \sum \lambda^j \otimes f_j(z) \in \HH(\Garr, \kf) \otimes \kf[\Vrr(G)],$$ where $\lambda^j = \lambda_1^{j_1} \cdots \lambda_r^{j_r}$. We define $\psi(z)$ as the coefficient for $\lambda_r^n$ in $\kf[\Vrr(G)]$. 
\item For $p > 2$ and $z \in H^{n,m}(G, \kf)$, $$u^*(z) = \sum \lambda^j x^i y^k \otimes f_{ijk}(z) \in \HH(\Garr, \kf) \otimes \kf[\Vrr(G)],$$ where $\lambda^j x^i y^k = \lambda_1^{j_1} \cdots \lambda_r^{j_r} x_1^{i_1} \cdots x_r^{i_r}y^k$. We define $\psi(z)$ as the sum of all the coefficients for $x_r^ly^k$ in $\kf[\Vrr(G)]$ such that $n = 2l+k$. 
\end{itemize}

For any $\gamma \in \Af^1$, the following diagram commutes, where $\gamma$ is acting on $\Garr$ and $\Vrr(G)$ as described in \ref{actions}.
$$\xymatrix{\Garr \times \Vrr(G) \ar[r]^{1 \times \gamma} \ar[d]_{\gamma \times 1} & \Garr \times \Vrr(G) \ar[r]^{u} & G \times \Vrr(G) \ar[d]^{\pi_1}\\ 
\Garr \times \Vrr(G) \ar[r]^{u} & G \times \Vrr(G) \ar[r]^{\pi_1} & G.}$$

To figure out the grading of $\psi(z)$ we pullback on the diagram above in the two possible ways and compare them. 

For $p = 2$ we get that $$\sum \gamma^{j_1+2j_2+\cdots+2^{r-1}j_{r}}\lambda^j \otimes f_j(z),$$ which implies that $\|\psi(z)\| = (n2^{r-1}, m-n|t|2^{r-1})$.

For $p >2$ we have that $$\sum \gamma^{2j_1+2pj_2+\cdots+2p^{r-1}j_r + 2pi_1+ \cdots + 2p^ri_r + p^r k}\lambda^j x^i y^k \otimes f_{ijk}(z).$$ In this case $\psi(z)$ is only homogeneous with respect to the cohomological degree and not bihomogeneous, but the bihomogeneous pieces of $\psi(z)$ have bidegree $(np^r, m-|t|p^rl-k|s|)$ where $2l+k = n$. 

It can be easily checked that $\psi$ is indeed an algebra map.
\end{proof}

\section{Detection properties and $F$-injectivity of $\psi$}

From \cite{SFB1} and \cite{SFB2} we have that for any infinitesimal group scheme of height $\leq r$, the ungraded version of $\psi$ is an $F$-isomorphism. A result of this type allows us to compute the cohomology of a group scheme $G$ (up to nilpotents) in a fairly straightforward way. We can do this by embedding $G$ into some $\Gl_n$ and using the defining equations of this embedding to describe $\kf[\Vr(G)]$ as a quotient of $\kf[\Vr(\Gl_n)]$.

Given the usefulness of such a result, it is a natural question to ask if we get a similar result in our case. In our quest to answer this question we compute some examples of the map $\psi$ and we show that they are in fact $F$-isomorphisms. Then, we define some detection properties of the cohomology of a gr-group scheme and study their relation. Finally, we (partially) answer the question for a class of gr-group schemes.

Recall that a map between $\kf$-algebras is \emph{$F$-isomorphism} if its kernel consists of nilpotent elements, and some $p$th power of every element in the target is actually in the image. We also recall from \ref{grcomcoh} that the cohomology of a finite gr-group scheme is graded commutative with respect to the total degree.

\subsection{Examples  of $\psi$ for some gr-group schemes}\label{expsi}

\begin{example}
For $\Garr$ and $p >2$: 

\begin{itemize}
\item $\HH(\Garr, \kf)$ is $F$-isomorphic to $\kf[x_1, \ldots, x_r, y]^{gr}$ where $\|x_i\| = (2, |u|p^i), |y| = (1,|v|)$.
 \item $\kf[\Vrr(\Garr)]$ is $F$-isomorphic to $\kf[x^0, \ldots, x^{r-1}, y^0]^{gr}$  where  $\|x^i\|= (2p^i, |u|-|t|p^i)$, and  $\|y^0\| = (p^r, |v|-|s|)$. 
\end{itemize}

By comparing bidegrees, $\psi(x_i) = (x^{r-i})^{p^i}, \psi(y) = y^0$.
\end{example}

\begin{example}
Recall the gr-group scheme $\St_1$ from \ref{A1} which has coordinate algebra $\kf[\St_1] = \frac{\F_2[\xi_1, \xi_2]^{gr}}{(\xi_1^4, \xi_2^2)}$. 
We have that  $\HH(\St_1, \F_2)$ is $F$-isomorphic to $\F_2[h_{10}, h_{20}]$, where $\|h_{10}\| = (1,1),$ and $\|h_{20}\|  = (1,3)$. 

We compute on \ref{vrst} that $\F_2[V_2^*(\St_1)]$ is $F$-isomorphic to $\F_2[x,y]^{gr}$ where $\|x\| = (2, 1-2|t|)$ and $\|y\| = (2, 3-2|t|)$. By comparing bidegrees, $\psi(h_{10}) = x$ and $\psi(h_{20}) = y$.

\end{example}

\begin{example}
Recall Wilkerson's counterexample from \ref{wilk1}, $\F_p[W_1]$. For $p >2$, we have that $\F_p[V_2^*(W_1)]$ is $F$-isomorphic to $\F_p[X^0_{13}, X^1_{12}, X^1_{13}, X^1_{15}]^{gr}$. From \cite[6.3]{Wi} we have that $\HH(W_1, \kf)$ is $F$-isomorphic to $$\F_p[b_{10}, b_{20}, b_{21}, b_{30}]^{gr}$$ where 
 \begin{itemize}
 \item $\|b_{10}\| = (2, 2p(p-1))$, 
 \item $\|b_{20}\| = (2, 2p(p^2-1))$, 
 \item $\|b_{21}\| = (2, 2p^2(p^2-1))$, and 
 \item $\|b_{30}\| = (2, 2p(p^3-1))$.
 \end{itemize}
 
 By comparing bidegrees, $\psi$ corresponds to the map that sends $b_{10} \mapsto (X_{12}^1)^p, b_{20} \mapsto (X_{13}^1)^p, b_{21} \mapsto (X_{13}^0)^{p^2},$ and $b_{30} \mapsto (X_{15}^1)^p$. 
 \end{example}

\begin{example}
Recall Wilkerson's counterexample from \ref{wilk2}: $\F_2[W_2]$, $\F_2[V_1^*(W_2)]$ is $F$-isomorphic to $\F_2[x_1, x_2, x_3, x_4, x_5]/(x_1x_4+x_2x_3)$, where $\|x_i\| = (1,i-|t|)$. From \cite[6.5]{Wi}, $\HH(W_2, \kf)$ is $F$-isomorphic to $$\frac{\F_2[z_1, z_2, z_3, z_4, z_5^2]}{(z_1z_4+z_2z_3)},$$ where $\|z_i\| = (1, i)$. We then have that $\psi(z_i) = x_i$ for $i <5$ and $\psi(z_5^2) = x_5^2$. 
\end{example} 

In \cite{SFB2}, they show that the ungraded $\psi$ is an $F$-isomorphism.  To show that $\psi$ is an $F$-monomorphism they first show that the cohomology of infinitesimal group schemes of height $\leq r$ satisfies a detection property. By a \emph{detection property} we mean some sort of Quillen-type result in which the cohomology of our object can be detected (up to nilpotent) by understanding the cohomology of some sub-class or restriction of this object. 

For infinitesimal group schemes, the detection property is one that detects the cohomology up to nilpotents by restricting to $1$-parameter subgroups. More precisely, in \cite[4.3]{SFB2} they show that $z \in H^n(G, \kf)$ is nilpotent if and only if for every field extension $K/\kf$ and every group scheme homomorphism over $K$, $\nu: \Gar \otimes K \to G \otimes K$, the cohomology class $\nu^*(z_K) \in H^n(\Gar \otimes K, K)$ is nilpotent. 

For gr-Hopf algebras, in \cite{Wi}, a detection property is defined in terms of elementary sub-Hopf algebras. A graded Hopf algebra as in \cite{Wi} satisfies the detection property if each nonnilpotent cohomology class has a nonzero restriction to at least one elementary sub-Hopf algebra. It is important to note that not all gr-Hopf algebras as in \cite{Wi} satisfy this detection property. Hence we do not get a Quillen-type result as we do in the ungraded case. 

Wilkerson's counterexamples in \cite{Wi} are example \ref{wilk1} (for the $p >2$ case)  and \ref{wilk2} above. He showed that they do not satisfy his detection property. Although we computed that for these gr-group schemes $\psi$ is an $F$-isomorphism, what happens is that even though they do not satisfy this detection property they satisfy some detection property (one like that of \cite{SFB2}), that suffices to guarantee that $\psi$ is an $F$-monomorphism. We define such detection property below and compare it to Wilkerson's. We conclude by proving that if a gr-group variety satisfies such detection property, then $\psi$ is an $F$-monomorphism.

\begin{definition}(From \cite[2.1.6]{P})
A gr-group scheme $G_E$ is said to be an \emph{elementary gr-group scheme} if  its coordinate ring $E$ is isomorphic (as a graded Hopf algebra) to a tensor product of graded Hopf algebras of the forms

\begin{itemize}
		\item  $\kf[t]^{gr}/(t^{2^n})$ for $p = 2$, and 
		\item $\kf[t]^{gr}/(t^{p^n})$ and $\kf[s]^{gr}$ where $|t|$ is even and $|s|$ is odd for $p>2$, 
\end{itemize}
where $t$ and $s$ are primitive elements. 

\end{definition}

\begin{rk}
Note that as usual, our definition includes the case in which $t$ is possibly of degree zero, while the definition in \cite{P} and \cite{Wi} is only for algebraically connected gr-group schemes.  \end{rk}

We now give the definition of two detection properties; one based on that of \cite{SFB2} the other on \cite{Wi}. 

\begin{definition}
Let $G$ be a gr-group scheme, we say that $G$ has the \emph{W-detection property} if $z \in H^{n,m}(G, \kf)$ is nilpotent if and only if for each elementary gr-subgroup scheme of $G$, $G_E$, its restriction to $H^{n,m}(G_E, \kf)$ is nilpotent. 
\end{definition}

Given a field  $K$, we can construct the \emph{graded field} $K[X, X^{-1}]$ ($K[X^{\pm}]$ for short), where $|X| = 1$ if $p = 2$, and $|X| = 2$ if $p >2$. Constructing this graded field is a useful tool that allows us to `grade' scalars on our field $K$. For instance, if we want to view a scalar $\lambda \in K$ as an element of degree $l \in \Z$, then we will identify $\lambda$ with $\lambda X^l \in K[X^{\pm}]$. Note that for $p = 2$, $|\lambda X^l| = l$, while for $p > 2$, $|\lambda X^l| = 2l$.

\begin{definition}\label{SFB}
Let $G$ be a gr-group scheme of height $\leq r$, we say that $G$ has the \emph{SFB-detection property} if $z \in H^{n,m}(G, \kf)$ is nilpotent if and only if for every field extension $K$ of $\kf$ and every gr-group scheme homomorphism over $K[X^{\pm}]$, $\nu: \Garr\otimes K[X^{\pm}] \to G \otimes K[X^{\pm}]$ the restriction of $z$ to $H^{n,m}(\Garr \otimes K[X^{\pm}], K[X^{\pm}])$  is nilpotent.
\end{definition}

Note that while the detection property in \cite{Wi} is  for algebraically connected cocommutative graded Hopf algebras and \cite{SFB2} detection property is for (ungraded) infinitesimal group schemes, our detection properties are constructed for gr-group schemes. 

One word on why we choose $K[X^{\pm}]$ for the SFB-detection property: for a field $\kf$, a graded field extension of $\kf$ may be one of the following; $K$ where $K$ is a field extension in the ungraded sense, or the graded field $K[X^{\pm}]$ where $K$ is a field extension in the ungraded sense. 

We state the result saying that all infinitesimal group schemes satisfy an ungraded detection property and we compare it to the SFB-detection property defined above. We will refer to this result as the ungraded SFB-detection property. 

\begin{theorem}(From \cite[4.3]{SFB2})\label{uSFB}
Let $G$ be an infinitesimal group scheme of height $\leq r$ over $\kf$. Let $z \in H^n(G, \kf)$, then $z$ is nilpotent if and only if for every field extension $K/\kf$ and every group scheme homomorphism over $K$, $\nu: \Gar \otimes K \to G \otimes K$, the cohomology class $\nu^*(z_K) \in H^n(\Gar \otimes K, K)$ is nilpotent. 
\end{theorem}

\begin{lemma}\label{homs}
Let $G$ be an evenly gr-group scheme and $K$ an ungraded field,  then $\Hom(\kf[G], K) = \Homg(\kf[G], K[X^{\pm}])$. 
\end{lemma}
\begin{proof}
We have the following identification, $f \in \Hom(\kf[G], K) \leftrightarrow \widehat{f} \in \Homg(\kf[G], K[X^{\pm}])$, where for $g$ homogeneous in $\kf[G]$, 
$f(g) = \lambda \leftrightarrow \widehat{f}(g) = \lambda X^{|g|}$. 
\end{proof}

\begin{proposition}\label{vr-vrr}
Let $G$ be an evenly gr-group scheme and let $K$ be an ungraded field, then $\Vr(G)$ and $\Vrr(G)$ have the same coordinate ring and $\Vr(G)(K) = \Vrr(G)(K[X^{\pm}])$.
\end{proposition}
\begin{proof}
Since $G$ is evenly graded we can embed $G$ into a $\Gl_I$ which is also evenly graded. In this case $\kf[\Vrr(\Gl_I)]$ is commutative (not just graded commutative) hence it is equal to $\kf[\Vr(\Gl_n)]$, by just forgetting the grading, where $n$ is the length of $I$. Similarly, $\kf[\Vrr(G)] = \kf[\Vr(G)]$ if we forget the grading. Hence $\Vr(G)$ and $\Vrr(G)$ have the same coordinate ring and $\Vr(G)(K) =\Vrr(G)(K[X^{\pm}])$ by \ref{homs}\end{proof}

Note that $\Vr(G)$ is not equal to $\Vrr(G)$, since they are representable functors on different categories even though they have the same coordinate ring. That is, $\Vr(G)(A) = \Homg(\kf[\Vr(G)], A)$ where $A$ is a commutative algebra, while $\Vrr(G)(B)= \Hom(\kf[\Vr(G)], B)$ where $B$ is a graded commutative algebra.

\begin{proposition}\label{uSFB-SFB}
Let $G$ be an evenly gr-group scheme. Then the ungraded SFB-detection property (see \ref{uSFB}) implies the SFB-property (see \ref{SFB}). Moreover, $G$ has the SFB-detection property.  
\end{proposition}
\begin{proof}
First note that by \ref{Ext} $H^*(G, \kf) = \HH(G,\kf)$ if we forget the internal grading. Let $z \in H^{n,m}(G, \kf)$ be nilpotent, then since $G$ has the ungraded SFB-detection property we get that for every  $\nu: \Garr \otimes K[X^{\pm}] \to G \otimes K[X^{\pm}]$,  $\nu^*(z) \in H^{n,m}(\Garr \otimes K[X^{\pm}], K[X^{\pm}])$ is nilpotent (by tensoring by $\kf[X^{\pm}]$). 

Let $\nu^*(z) \in H^{n,m}(\Garr \otimes K[X^{\pm}], K[X^{\pm}])$ be nilpotent for every $\nu: \Garr \otimes K[X^{\pm}] \to G \otimes K[X^{\pm}]$. By \ref{vr-vrr} it follows that  $z$ is nilpotent.  
\end{proof}

\begin{proposition}
Let $G$ be a finite gr-group scheme, the W-detection property implies the SFB-detection property. In particular, elementary gr-group schemes have the SFB-detection property. 
\end{proposition}
\begin{proof}
Without loss of generality, we can assume that $G = G_E$ is an elementary gr-group scheme of height $\leq r$. We will prove that $G_E$ has SFB-detection property. 
For simplicity let $p >2$; the $p = 2$ case follows in the same way. The coordinate ring of $G_E$ is  
$$E =\frac{\kf[t_1, \ldots, t_n, s_1, \ldots, s_m]^{gr}}{(t_1^{p^{r_1}}, \ldots, t_n^{p^{r_n}})},$$
where the $t_i$'s and $s_i$'s are primitive, $|t_i|$ is even, and $|s_i|$ is odd. 

Let $A = K[X^{\pm}]$, consider the map from $A[G_E] \to A[\Garr]$ given by $t_i \mapsto t^{p^{r-r_i}}X^{k_i}$ and $s_i \mapsto sX^{m_i}$ where $k_i = (|t_i|-p^{r-r_i}|t|)/2$ and $m_i = (|s_i|-|s|)/2$; which are integers. This yields the map $\nu_E: \Garr \otimes A \to G_E \otimes A$. 
From \ref{HGar}, the cohomology of $G_E$ and $\Garr$ are:
$$\HH(G_E, \kf) = \bigotimes_{i = 1}^n \big(\kf[x_{i1}, \ldots, x_{ir_i}, y_i, \lambda_{i1}, \ldots, \lambda_{ir_i}) \big]^{gr}, \mbox{ and}$$
$$\HH(\Garr, \kf) = \kf[x_1, \ldots, x_r, y, \lambda_1, \ldots, \lambda_r]^{gr}.$$

Then $\nu_E^*:\HH(G_E \otimes A, A) \to \HH(\Garr \otimes A, A)$ is given by $x_{ij} \mapsto  x_{r-r_i+j}X^{p^jk_i}$, $y_i \mapsto yX^{m_i}$, and $\lambda_{ij} \mapsto \lambda_jX^{l_i}$ where $l_i = (|t_i|-|t|)/2$. 
Then $\nu_E^*$ sends nonnilpotent elements to nonnilpotent elements, hence $G_E$ has the SFB-detection property. 
\end{proof}

\begin{rk}
In \cite{Wi} Wilkerson showed that examples $\F_p[W_1]$ and $\F_2[W_2]$ in \ref{wilk1} and \ref{wilk2}, respectively, do not satisfy the W-detection property. 
However, since $W_1$ and $W_2$ are evenly gr-group schemes, by \ref{uSFB} and \ref{uSFB-SFB} they satisfy the SFB-detection property. 
\end{rk}

\begin{lemma}\label{detec}
Let $G$ be a finite gr-group variety of height $\leq r$ with the SFB-property, then for any $z \in H^{n,m}(G, \kf)$, $z$ is nilpotent whenever $\psi(z) \in \kf[\Vrr(G)]$ is nilpotent. 
\end{lemma}

\begin{proof}
Our proof is based on that in \cite[5.1]{SFB2}. Let $z \in H^{n,m}(G, \kf)$ with $\psi(z)$ nilpotent. Since $G$ has the SFB-property, it is enough to prove that, for all field extensions $K$ of $\kf$ and every gr-group scheme homomorphism $\nu:\Garr \otimes A \to G \otimes A$, $\nu^*(z_A) \in H^{n,m}(\Garr\otimes A, A)$ is nilpotent, where $A = K[X^{\pm}]$. Without loss of generality, let $K$ be algebraically closed. 
 Let $\overline{\nu^*(z_A)}  \in A[x_1, \ldots, x_r, y]^{gr}$ be the element in $H^{n,m}(\Garr \otimes A,A)_{red}$ (the reduce cohomology) corresponding to $\nu^*(z_A) \in H^{n,m}(\Garr \otimes A,A)$. Then $\overline{\nu^*(z_A)}$ equals
$$\sum_{\substack{2(i_1+ \cdots + i_r)+ j = n \\\ |t|(pi_1+p^2i_2 + \cdots + p^ri_r) +|s|j = m}} a_{(\underline{i}, j)}x_1^{i_1}x_2^{i_2} \cdots x_r^{i_r}y^j,$$

where $(\underline{i}, j) = (i_1, \ldots, i_r, j)$. We will show that $\overline{\nu^*(z_A)}  = 0$ and since $G$ has the SFB-detection property, it will follow that $z$ is nilpotent. 

Let $(\underline{c}, d) = (c_1, \ldots, c_r, d)$ be an $r+1$-tuple  in $A$, where $|c_i| = |t|(1-p^{i-1})$ and $|d| = 0$. Note that $|c_i|$ depends only on $i$ and not on the choice of $c_i$. We can write $c_i = \tilde{c}_iX^{|c_i|}$ and $(\tilde{c}, d) = (\tilde{c}_1, \ldots, \tilde{c}_r, d)$ an $r+1$-tuple in $K$. 

Let $E$ be the elementary gr-group scheme over $A$ with coordinate ring 
$$A[E] = A[t_1, \ldots, t_r, s]^{gr}/(t_1^{p^r}, \ldots, t_r^{p^r}),$$
 where $|t_i| = |t|$ and $s$ is graded as in $\Garr$.   Let $\gamma_{(\underline{c}, d)}:\Garr \otimes A \to \Garr \otimes A$ denote the composition 

$$\xymatrixcolsep{2pc}\xymatrix{\Garr \otimes A \ar[rr]^-{(\Delta^{ev})^r \times \Delta^{od}} & \hspace{.7cm}  &E \otimes A \ar[r]^{F_{(\underline{c}, d)}} &  E \otimes A \ar[r]^<<<<<m  & \Garr \otimes A},$$

where
\begin{itemize}
\item  $m^*: A[\Garr] \to A[E]$ is given by $t \mapsto \Delta^{r-1}(t) \otimes 1$, $s \mapsto 1 \otimes \cdots \otimes 1 \otimes s$, followed by the map $(t \mapsto t_1) \otimes \cdots \otimes (t \mapsto t_r) \otimes (s \mapsto s)$ and followed by multiplication. 
\item $F_{(\underline{c}, d)}^*: A[E] \to A[E]$ is given by $t_i \mapsto c_it_i^{p^{i-1}}, s \mapsto ds$, and finally
\item $((\Delta^{ev})^r \times \Delta^{od})^*: A[E] \to A[\Garr]$ is given by $t_i \mapsto t$, and $s \mapsto s$. 
\end{itemize}

It follows that on cohomology
$$\gamma^*_{(\underline{c}, d)}(x_i) = c_1^{p^{i-1}}x_i + c_2^{p^{i-1}}x_{i+1} + \ldots + c_{r-i+1}^{p^{i-1}}x_{r},\mbox{ and } \gamma^*_{(\underline{c}, d)}(y) = dy.$$

Recall from \ref{HG-KV} that $u:\Garr \otimes \kf[\Vrr(G)] \to G \otimes \kf[\Vrr(G)]$ is defined to be the identity map $1 \in \Vrr(G)(\kf[\Vrr(G)])$. The universality of $u$ gives that for any graded commutative $\kf$-algebra $A$ and any $\omega: \Garr \otimes A \to G \otimes A$ the following diagram commutes.

$$\xymatrix{\HH(G,\kf) \ar[d] \ar@/_10pc/[dddd]_{u^*} \ar[r]^{=} & \HH(G, \kf) \ar[d] \ar@/^10pc/[dddd]^{\omega^*} \\
\HH(G, \kf) \otimes \kf[\Vrr(G)] \ar[d]^{=} & \HH(G,\kf) \otimes A \ar[d]^{=} \\ 
\HH(G \otimes \kf[\Vr^*(G)]) \ar[d] & \HH(G \otimes A, A) \ar[d] \\
\HH(\Garr \otimes \kf[\Vr^*(G)], \kf[\Vr^*(G)]) \ar[d]^= & \HH(\Garr \otimes A, A) \ar[d]^= \\
\HH(\Garr, \kf) \otimes \kf[\Vrr(G)] \ar[d] & \HH(\Garr, \kf) \otimes A \ar[d]\\
\kf[\Vrr(G)] \ar[r] & A
}$$

Now $\psi(z)$ is given by following $z \in \HH(G,\kf)$ through the left side of the diagram and then looking at the sum of all the coefficients for $x_r^ly^k$ such that $n = 2l+k$ in $\kf[\Vrr(G)]$.

Let $\nu' = \nu \circ \gamma_{(\underline{c}, d)}: \Garr \otimes A \to G \otimes A$, then $\overline{{\nu'}^*(z_A)} \in H^{n, m}(\Garr \otimes A, A)_{red}$ equals
$$\sum_{\substack{2(i_1+ \cdots + i_r)+ j = n \\\ |t|(pi_1+p^2i_2 + \cdots + p^ri_r) +|s|j = m}} a_{(\underline{i},j)}(c_1x_1+ \cdots + c_rx_r)^{i_1} \cdots (c_1^{p^{r-1}}x_r)^{i_r}(dy)^j.$$

If $\psi(z)$ is nilpotent then it follows from the diagram, in the case of $\omega= \nu'$, that the sum of all the coefficients of $x_r^ly^k$ such that $n = 2l+k$ for $\overline{{\nu'}^*(z_A)}$  is zero. This sum is equal to 
\begin{equation*}
\sum_{\substack{2(i_1+ \cdots + i_r)+ j = n \\\ |t|(pi_1+p^2i_2 + \cdots + p^ri_r) +|s|j = m}} a_{(\underline{i}, j)}(c_r)^{i_1}(c_{r-1}^p)^{i_2} \cdots (c_1^{p^{r-1}})^{i_r}d^j.
\end{equation*}

We can rewrite this equation as 
\begin{equation*}
\sum_{\substack{2(i_1+ \cdots + i_r)+ j = n \\\ |t|(pi_1+p^2i_2 + \cdots + p^ri_r) +|s|j = m}} \tilde{a}_{(\underline{i}, j), l}(\tilde{c}_r)^{i_1}(\tilde{c}_{r-1}^p)^{i_2} \cdots (\tilde{c}_1^{p^{r-1}})^{i_r}d^j X^{q(\underline{i}, l)},
\end{equation*}

where $a_{(\underline{i}, j)} = \sum_l \tilde{a}_{(\underline{i}, j), l} X^l$ and  $q(\underline{i}, l) = {i_1|c_r|+pi_2|c_{r-1}| + \cdots + p^{r-1}i_r|c_1| + l}$. 

Let 
$$f_{\underline{i}, l}(\underline{c}, d) = \sum_{\substack{2(i_1+ \cdots + i_r)+ j = n \\\ |t|(pi_1+p^2i_2 + \cdots + p^ri_r) +|s|j = m}} \tilde{a}_{(\underline{i}, j), l}(\tilde{c}_r)^{i_1}(\tilde{c}_{r-1}^p)^{i_2} \cdots (\tilde{c}_1^{p^{r-1}})^{i_r}d^j,$$

then each homogeneous term $f_{\underline{i},l}(\underline{c},d) = 0$ for any choice of $r+1$-tuple $(\underline{c}, d) \in K$. Hence by the Nullstellensatz $f_{\underline{i}, j} \equiv 0$ as polynomial with coefficients in $K$ and $\tilde{a}_{(\underline{i},j),l} = 0$. Therefore,  $\overline{\nu^*(z_A)}  = 0$ and since $G$ has the SFB-detection property, it follows that $z$ is nilpotent as desired.
\end{proof}

Our main theorem of this chapter is now a direct consequence of \ref{detec}. We also get some corollaries from it. 

\begin{theorem}\label{F-mono}
Let $G$ be a finite gr-group variety of height $\leq r$. If $G$ has the SFB-detection property, then $\psi: \HH(G, \kf) \to \kf[\Vrr(G)]$ is an $F$-monomorphism, that is, its kernel consists of nilpotent elements.
\end{theorem}

\begin{corollary}
Let $G$ be an evenly gr-group variety of height $\leq r$, then $\psi: \HH(G, \kf) \to \kf[\Vrr(G)]$ is an $F$-monomorphism.\end{corollary}

\begin{corollary}
Let $G$ be an elementary gr-group scheme of  height $\leq r$, then $\psi: \HH(G, \kf) \to \kf[\Vrr(G)]$ is an $F$-monomorphism.\end{corollary}

\appendix 
\section{Cohomology of graded algebras}\label{cohoap}

For reference and completeness we provide some result and definitions regarding the cohomology of gr-group schemes. Let $R$ be a graded ring.

\begin{definition}
 Let $M$ and $N$ be graded $R$-modules, a map $f: M \to N$ is a \textit{graded map of degree $n$} if $f(M_i) \subset N_{i+n}$ for all $i$. 
\end{definition}

\begin{definition}\label{sus}
 The $i$-\textit{suspension} of a graded $R$-module $M$ is defined to be the graded module $M(i)$ given by $M$ so that $M(i)_j = M_{i+j}$. 
\end{definition}

\begin{definition}\label{uHom}
Let $M$ and $N$ be graded $R$-modules. Let $\uHom_R(M,N)_n$ denote the subgroup of $\Hom_R(M, N)$ consisting of maps of degree $n$. We define $\uHom_R(M,N) = \oplus_i \uHom_R(M,N)_i$. 
\end{definition}

\begin{definition}
Since $\uHom_R(\_,\_)$ is left exact we can defined the right derived functor denoted by $\uExt_R^n(\_,\_)$.
\end{definition}

\begin{proposition}(\cite[2.4.4]{NV})
 If $M, N$ are graded modules over $R$ and $M$ is finitely generated then $$\uHom_R(M, N) = \Hom_R(M,N).$$
\end{proposition}

\begin{definition}
 A graded module $M$ is \textit{left gr-Noetherian} if $M$ satisfies the ascending chain condition for graded left $R$-submodules. 
\end{definition}

\begin{proposition}[P. Samuel] A $\Z$-graded ring $R = \oplus_{i \in \Z} R_i$, is gr-Noetherian if and only if $R_0$ is a Noetherian ring and $R$ 
is finitely generated as an $R_0$ algebra. 
\end{proposition}

\begin{proposition}(From \cite[2.4.7]{NV})\label{Ext}
 If $R$ is left gr-Noetherian and $M$ is finitely generated then, for every $n \geq 0$ $$\uExt_R^n(M,N) = \Ext^n_R(M,N).$$
\end{proposition}

\begin{definition}
 A graded $R$-module $P$ is \textit{gr-projective} if it is projective as an object in the category of graded $R$-modules. That is, if $P$ satisfies
the universal lifting property of projective modules, where the maps are graded maps between graded modules. It can be shown (with the usual proofs following
verbatim), that $P$ is gr-projective if and only if $P$ is a direct summand of a graded free module. 
\end{definition}

\begin{proposition}
 A graded module $P$ is gr-projective if and only if $P$ is graded and projective. 
\end{proposition}

\begin{proof}
 One direction is clear, if $P$ is gr-projective, then $P$ is a direct summand of a gr-free, forgetting the grading on $F$ it is then a direct
summand of a free module, hence it is projective and also graded. 

On the other hand if $P$ is a graded module which is also a projective module, then let $f: M \twoheadrightarrow N$ be a graded
surjective map and $g: P \to N$ a graded map, then since $P$ is projective there exists a map $h: P \to M$ (not necessarily graded) such that
the following diagram commutes. 

$$\xymatrix{& M \ar@{->>}[d]^f\\ P \ar@{-->}[ur]^{\exists h} \ar[r]^g & N}$$

We construct a graded map $\widetilde{h}: P \to M$ from $h$, such that the diagram will commute if we substitute $h$ with $\widetilde{h}$.
Consider the group homomorphisms $\pi_i:M \to M_i$ which are the projections of $M$ onto its degree $i$ part $M_i$
and define the group homomorphisms $\widetilde{h_i}: P_i \to M_i$ to be $\widetilde{h_i} = \pi_i \circ h |_{P_i}$. We then define
$\widetilde{h} = \oplus_i \widetilde{h_i}: P \to M$.

The map $\widetilde{h}$ is degree preserving. Let $p_i \in P_i$, then $\widetilde{h}(p_i) = \widetilde{h_i}(p_i) = \pi_i \circ(h(p_i)) $. If we write $h(p_i) = \sum_j m_j$, then $\pi_i(h(p_i)) = m_i$. 

The map $\widetilde{h}$ is in fact an $R$-module map, since $\widetilde{h}$ is defined as composition of group homomorphims $\widetilde{h}(p+q) = \widetilde{h}(p)+\widetilde{h}(q).$
Since $\widetilde{h}$ distributes with the sum, to check if elements of $R$ factor out it is enough to check on homogeneous elements. Let $r \in R_i$ and 
$p \in P_j$, then $rp \in P_{i+j}$ hence $\widetilde{h}(rp) = \pi_{i+j} \circ h(rp) = \pi_{i+j} (r h(p)) = r \pi_j(h(p)) = r \widetilde{h}(p)$. 

To show $g = f \circ \widetilde{h}$ we again assume $p \in P_i$, homogeneous, then if $h(p) = \sum_j m_j$, then
$f \circ \widetilde{h}(p) = f(\pi_i(h(p))) = f(m_i)$; on the other hand $f \circ h(p) = g(p)$ which gives that 
$f(h(p)) = f(\sum_j m_j) = \sum_j f(m_j) = f(m_i)$ since $g$ and $f$ are graded maps.  
\end{proof}

\begin{rk}(From \cite[2.2]{NV})
Note that the proposition above is not true if we replace gr-projective with gr-free as we may have a module which is free with respect to a non-homogeneous basis and not free with respect to a graded basis. For example let  $R = \Z \times \Z$ with trivial grading and 
$M$ the graded $R$ module where $M_0= \Z \times {0}$, $M_1 = {0} \times \Z$ and $M_i = 0$ for $i \neq 0,1$, then $M$ cannot have a homogeneous basis. Hence, gr-free is a stronger property than graded and free.
\end{rk}

\begin{proposition}
 If $Q$ in $R$-gr is injective when considered as an ungraded module, then $Q$ is gr-injective. 
\end{proposition}

\begin{proof}
 From \cite[2.3.2]{NV}. The proof is very similar to the one above about projective module except that a gr-injective module, may not be 
injective as an ungraded module, so we do not get an if and only if statement.\end{proof}
We get the following corollary. 
\begin{corollary}
 Let $P$ be an $R$-module, then $P$ is a gr-projective if and only if $P$ is projective as an object in the category where the objects 
are graded modules and morphisms are $\uHom_R(\_, \_)$. 
\end{corollary}
We can do similar argument about gr-flat modules. 
\begin{definition}
A graded module $M$ is gr-flat if the functor $\_ \otimes_{\Z} M$ is exact. \end{definition}
We can show that $M$ is gr-flat if and only if $M$ is graded and flat, c.f. \cite[2.12.11]{NV}. 

Given a graded algebra $A$, we want to compute its cohomology, which will turn out to be a bigraded ring, 
that is $\HH(A, \kf)$. 
Since a gr-projective module is the same as a projective and graded module, a gr-projective resolution of $\kf$ as a graded $A$-module will also be a
projective resolution of $\kf$ as an $A$-module. 
We then apply $\Hom_{A}(\_, \kf)$ to compute $\Ext^*_{A}(\kf, \kf)$ in order to compute the usual cohomology of $A$, without keeping 
track of the grading. 

If instead we want to compute the graded cohomology, we would need to apply the functor $\uHom_{A}(\_, \kf)$. 
 Let $$\xymatrix{ \cdots \ar[r]^{d_2} & P_1 \ar[r]^{d_1} & P_0 \ar[r]^{d_0} & \kf \ar[r] & 0}$$
be a graded projective resolution of $\kf$ as an $A$-gr-module. We then apply the functor $\uHom_A(\_, \kf)$ to the resolution to get complex
$$\xymatrix{0 \ar[r] & \uHom_A(P_0, \kf) \ar[r]^{\delta_0} & \uHom_A(P_1, \kf) \ar[r]^{\delta_1} & \uHom_A(P_2, \kf) \ar[r]^{\delta_2} & \cdots}$$
where $\delta_i: \uHom_A(P_i, \kf) \to \uHom_A(P_{i+1}, \kf)$ is given by $\delta_i(f)(p)= f(d_{i+1}(p))$ for $p \in P_{i+1}$ and $f \in \uHom_A(P_i, \kf)$. 

The \textit{$n^{th}$ graded cohomology} is defined by $H^n(A, \kf) = \frac{\ker(\delta_n)}{\im(\delta_{n-1})}$.

If $A$ is  gr-Noetherian then by \ref{Ext} we would get the same cohomology as the ungraded one. 

\begin{proposition}\label{les}
 If $$\xymatrix{\cdots \ar[r]^{f_3} & M_2 \ar[r]^{f_2} & M_1 \ar[r]^{f_1} & M_0}$$ is a sequence of graded maps where each $f_i$ may be of 
a nonzero degree, then this sequence is equivalent to one where the maps are all graded of zero degree.  
\end{proposition}
\begin{proof}
 This follows from the fact that the suspension \ref{sus} is a functor and that the suspension of a gr-projective module is still gr-projective. 
We also use that $\uHom_R(M, N)_i = \Hom_{R-gr}(M, N(i)) = \Hom_{R-gr}(M(-n), N)$. 
\end{proof}

\begin{rk}
 The above observations tell us that given a graded algebra $A$, we can assume our category to be the one of graded $A$ modules,
where the morphisms are graded maps of various degrees. Then when computing the cohomology, the gr-projective resolution can consist of maps
of degree zero and to the resolution we apply the functor $\uHom_A(\_, \kf)$. 
\end{rk}

\begin{proposition}
 Let $A$ be a graded algebra then $\HH(A,\kf) = \uExt^{*,*}_A(\kf, \kf)$ is a bigraded $\kf$-module.
\end{proposition}

\begin{proof}
Let  $$\xymatrix{ \cdots \ar[r]^{d_2} & P_1 \ar[r]^{d_1} & P_0 \ar[r]^{d_0} & \kf \ar[r] & 0}$$
be a graded projective resolution of $\kf$ as an $A$-gr-module. By \ref{les} we can assume that the maps $d_i$ are graded of degree zero.
Let $f = \sum_s f_s \in \uHom_A(P_n, \kf)$, where $f_s \in \uHom_A(P_n,\kf)_s$, that is $f_s:P_n \to \kf$ is a map of degree $s$.
We want to show that if $f \in \ker(\delta_n)$ then each $f_s$ is in the kernel as well. 
Hence $\ker(\delta_n)$ is generated by homogeneous elements of $\uHom_A(P_n, \kf)$. 

For $p \in P_{n+1}$,
$\delta_n(f)(p) = f(d_{n+1}(p)) = \sum_s f_s(d_{n+1}(p))$, and each $f_s \circ d_{n+1}$ is a map of degree $s$. Now $f \in \ker(\delta_n)$
implies that $\delta_n(f)(p) = 0$ for all $p \in P_{n+1}$. 

On a homogeneous element $p$ of $P_{n+1}$, $$\delta_n(f_s)(p) = \left\{
	\begin{array}{ll}
		\lambda_s & \mbox{if }  |p| = -s \\
		0  & \mbox{otherwise }
	\end{array}
\right.$$
Hence $$\delta_n(f)(p) = \sum_s \delta_n(f_s)(p) = \left\{
	\begin{array}{ll}
		\lambda_s  & \mbox{if }  |p| = -s \\
		0  & \mbox{otherwise }
	\end{array}
\right.$$

This gives that if $f \in \ker(\delta_n)$ then $\lambda_s = 0$, thus $\delta_n(f_s)(p)=0$ for each homogeneous elements; therefore $f_s \in \ker(\delta_n)$ 
as well. 

We want to show that $\im(\delta_n)$ is also generated by homogeneous elements of $\uHom_A(P_{n+1}, \kf)$. 
In this case $f = \sum f_s \in \uHom_A(P_n, \kf)$ and $\delta_n(f)(p) = \sum \delta_n(f_s)$ and clearly each $\delta_n(f_s)$ is homogeneous 
of degree $s$. Therefore we have $\uExt^n_A(\kf, \kf) = \bigoplus_s \uExt^{n,s}_A(\kf, \kf)$. That is, a $\zeta \in H^n(A, \kf)$ corresponds to a map $\widehat{\zeta}: P_n \to \kf$ of degree $s$. 
\end{proof}

Similarly we can define the graded cohomology of $A$ with coefficients in $M$ by $\HH(A,M)$ where $M$ is a graded $A$-module.
  
We now give other characterizations for $\HH(A,\kf)$.

\begin{definition}
Given two graded modules $M$, $N$ a \textit{graded $n$-extension of $M$ to $N$} is an exact sequence of graded maps of various degrees of the form
$$\xymatrix{0 \ar[r] & N \ar[r] & M_{n-1} \ar[r] & \cdots \ar[r]& M_0 \ar[r] & M \ar[r] & 0 }$$
\end{definition}

\begin{definition}
 Two graded $n$-extensions are \textit{equivalent} if there exist graded maps, of possibly nonzero degree,
$h_0, \ldots, h_{n-1}$ such that the diagram commutes.

$$\xymatrix{0 \ar[r] & N \ar[r]^{g_n} \ar[d]^=  & M_{n-1} \ar[d]^{h_{n-1}} \ar[r]^{g_{n-1}} & \cdots \ar[r] \ar[d] & M_0 \ar[r]^{g_0} \ar[d]^{h_0} & M \ar[r] \ar[d]^= & 0 \\
0 \ar[r] & N \ar[r]^{g'_n} & M_{n-1}' \ar[r]^{g'_{n-1}} & \cdots \ar[r]& M_0' \ar[r]^{g'_0} & M \ar[r] & 0 }$$
Two equivalent extension will have as a consequence that  $\sum |g_i| = \sum |g'_i|$. Using that it 
can be shown that two equivalent extensions correspond to the same element in $\uExt^{n,s}(M,N)$, where $s = -\sum |g_i|$. We complete this to an equivalence relation by symmetry and transitivity in the usual way. 
\end{definition}

\begin{proposition}
An equivalence class of graded $n$-extension of $M$ to $N$ of the form 
$$\xymatrix{0 \ar[r] & N \ar[r]^{g_{n}} & M_{n-1} \ar[r]^{g_{n-1}} & \cdots \ar[r]^{g_1} & M_0 \ar[r]^{g_0} & M \ar[r] & 0 }$$
corresponds to an element in $\uExt^{n, s}(M,N)$ where $s = -\sum_{i = 0}^n |g_i|$.
\end{proposition}

\begin{proof}
We use the fact that the gr-projective resolution is in fact a projective resolution and we can get maps from the resolution to the extension via the usual 
discussion of extensions and projective resolutions in the ungraded case.
$$\xymatrix{\cdots \ar[r] & P_n \ar[r]^{d_n} \ar[d]^{f_n} & P_{n-1} \ar[r]^{d_{n-1}} \ar[d]^{f_{n-1}} & \cdots \ar[r] & P_0 \ar[r]^{d_0} \ar[d]^{f_0}  & M \ar[r] \ar[d]^= & 0\\
0 \ar[r] & N \ar[r]^{g_{n}} & M_{n-1} \ar[r]^{g_{n-1}} & \cdots \ar[r]^{g_1} & M_0 \ar[r]^{g_0} & M \ar[r] & 0 }$$

It can be checked that these maps can be assumed to be graded and the last map $f_n$ will then have degree $s = -\sum_{i = 0}^n |g_i|$ 
and will correspond to an element in $\uExt^{n,s}(M, N)$. 
 \end{proof}

We relate an element in $ \zeta \in \uExt^{n,s}(M, N)$ with a map $\widehat{\zeta} \in \Omega^n(M) \to N$. Here we will be working on the stable category. 

\begin{proposition}
 Let $\zeta \in \uExt^{n, s}(M, N)$, then there is a map of degree $s$, $\widehat{\zeta}: \Omega^n(M) \to N$ 
corresponding to $\zeta$, if two such maps exist they correspond to equivalent extensions. Where $\Omega^n(M)$ is well-defined up to projective summands. 
\end{proposition}

\begin{proof}
 Recall that given a projective resolution of $M$, 
$$\xymatrix{\cdots \ar[r] & P_n \ar[r]^{d_n} & P_{n-1} \ar[r]^{d_{n-1}} & \cdots \ar[r] & P_0 \ar[r]^{d_0}  & M \ar[r] & 0,}$$

 $\Omega^n(M) = \ker(d_{n-1}) = \im(d_n) \subset P_{n-1}$. An element $\zeta \in \uExt^{n,s}(M,N)$ 
corresponds to an extension, where the following diagram commutes. 
$$\xymatrix{\cdots \ar[r] & P_n \ar[r]^{d_n} \ar[d]^{\zeta} & P_{n-1} \ar[r]^{d_{n-1}} \ar[d]^{f_{n-1}} & \cdots \ar[r] & P_0 \ar[r]^{d_0} \ar[d]^{f_0}  & M \ar[r] \ar[d]^= & 0\\
0 \ar[r] & N \ar[r]^{g_{n}} & M_{n-1} \ar[r]^{g_{n-1}} & \cdots \ar[r]^{g_1} & M_0 \ar[r]^{g_0} & M \ar[r] & 0 }$$

Let $id(-|g_n|): N(|g_n|) \to N$ the map of degree $-|g_n|$ that suspends back $N(|g_n|)$ to $N$, that is $id(-|g_n|)(N(|g_n|) = N$. 
We can then define $\widehat{\zeta} = id(-|g_n|) \circ f_{n-1}|_{\Omega^n(M)}$. Since $|f_{n-1}| = -\sum_{i = 0}^{n-1} |g_i|$, then 
$|\widehat{\zeta}| = -|g_n| -\sum_{i = 0}^{n-1} |g_i| = -\sum_{i = 0}^{n} |g_i| = s$ as desired. The last statement of the 
proposition can be proven as in the ungraded case. 
\end{proof}

There are several ways of describing the ring structure on $\HH(A, \kf)$. One way is doing the \emph{Yoneda splice} of an extension. We use the Yoneda splice to prove \ref{HHring}. 

\begin{definition}
Let $A$ be a graded algebra and let $M, M'$ and $M''$ be graded $A$-modules. Let $\zeta \in \uExt^{n,s}(M, M')$ and $\eta \in \uExt^{m,r}(M', M'')$, then $\zeta$ and $\eta$ correspond to the following extensions:
$$\xymatrix{\zeta: 0 \ar[r] & M' \ar[r]^{g_n} & M_{n-1} \ar[r]^{g_{n-1}} & \cdots \ar[r]^{g_1} & M_0 \ar[r]^{g_0} & M \ar[r] & 0\\
\eta: 0 \ar[r] & M'' \ar[r]^{g'_m} & M'_{m-1} \ar[r]^{g'_{m-1}} & \cdots \ar[r]^{g'_1} & M'_0 \ar[r]^{g_0} & M' \ar[r] & 0.}$$

Then $s = - \sum |g_i|$ and $r = - \sum |g'_i|$. The \emph{Yoneda splice} is defined as the  sequence given by 

\begin{center}
\begin{tikzpicture}
\node (A) at (0,0) {$M'$};
\node (B) at (-1,1) {$M'_0$};
\node (C) at (1,1) {$M_{n-1}$};
\node (D) at (-2.5,1) {$\cdots$};
\node (E) at (-4,1) {$M'_{m-1}$};
\node (F) at (-5.5,1) {$M''$};
\node (G) at (-6.5,1) {$0$};
\node (H) at (2.5,1) {$\cdots$};
\node (I) at (4,1) {$M_0$};
\node (J) at (5.5,1) {$M$};
\node (K) at (6.5,1) {$0$};
\node (L) at (-1,-1) {$0$};
\node (M) at (1,-1) {$0$};
\path[->] (G) edge (F);
\path[->] (F) edge node[above]{$g'_m$}(E);
\path[->] (E) edge node[above]{$g'_{m-1}$}(D);
\path[->] (D) edge node[above]{$g'_1$}(B);
\path[->] (B) edge node[left]{$g'_0$}(A);
\path[->] (A) edge node[right]{$g_n$}(C);
\path[->] (C) edge node[above]{$g_{n-1}$}(H);
\path[->] (H) edge node[above]{$g_1$}(I);
\path[->] (I) edge node[above]{$g_0$}(J);
\path[->] (J) edge (K);
\path[->] (L) edge (A);
\path[->] (A) edge (M);
\path[->] (B) edge (C);
\end{tikzpicture}
\end{center}

Then the Yoneda splice gives an element $\eta \circ \zeta \in \uExt^{n+m, s+r}(M, M'')$.  
\end{definition}

Note that the Yoneda splice makes $\uExt^{*,*}(M, M)$ into a ring, hence we get the following result. 

\begin{theorem}\label{HHring}
 Let $A$ be a graded algebra then $\HH(A,\kf) = \uExt^{*,*}_A(\kf, \kf)$ is a bigraded ring with $H^{n,s}(A,\kf) = \uExt^{n,s}_A(\kf, \kf)$, 
where $n$ is the cohomological degree and $s$ is the internal degree.
\end{theorem}
Let $A$ be a gr-Hopf algebra and $C$ and $D$ be graded $A$-modules. Then $(C \otimes D)$ is a graded $A$-module as well, where $(C \otimes D)_n = \bigoplus_{i+j = n} C_i \otimes D_j$. The action of $A$ on $C\otimes D$ is defined the following way: for $a \in A$, write $\Delta(a) = \sum a_1 \otimes a_2$, then $a \cdot (x \otimes y) = \Delta(a)(x \otimes y) = \sum (-1)^{|a_2||x|} a_1x \otimes a_2 y$.

\begin{lemma}\label{comhom}
Let $A$ be a cocommutative gr-Hopf algebra and $C$ and $D$ be graded $A$-modules. Let $T: C\otimes D \to D \otimes C$ be the map given by $x \otimes y \mapsto (-1)^{|x||y|}y \otimes x$ for $x$ and $y$ homogeneous elements. Then $T$ is a graded $A$-module morphism. 
\end{lemma}
\begin{proof}
Recall that $A$ is cocommutative if for any homogeneous element $a$, $\Delta(a) = \sum a_1 \otimes a_2 = \sum (-1)^{|a_1||a_2|} a_2 \otimes a_1$. We just need to write out and check the sign conventions on homogeneous elements. 
\begin{eqnarray*}
T(a \cdot (x \otimes y)) & = & T( \sum (-1)^{|a_2||x|} a_1 x \otimes a_2 y)\\
&=& \sum (-1)^{|a_2||x|+|a_1x||a_2y|} a_2y \otimes a_1x \\
&=& \left(\sum (-1)^{|a_1||a_2|}a_1 \otimes a_2 \right) (-1)^{|x||y|} y \otimes x \\
&=& \Delta(a) T(x \otimes y) \\
&=& a \cdot T(x \otimes y)
\end{eqnarray*}\end{proof}

\begin{definition}\label{prodcom}
Let $\bld{C}$ and $\bld{D}$ be complexes of graded $A$-modules. The \emph{tensor product of complexes} $\bld{C} \otimes \bld{D}$ is defined as, $(\bld{C} \otimes \bld{D})_n = \bigoplus_{i + j} C_i \otimes D_j$, where $\delta^{\bld{C} \otimes \bld{D}}(x \otimes y) = \delta^{\bld{C}}(x) \otimes y +(-1)^{\tot(x)} x\otimes \delta^{\bld{D}}(y)$. A homogeneous element $x \otimes y \in (\bld{C} \otimes \bld{D})_n$ corresponds to $x \in C^i_l, y \in D^j_k$ where $i+j = n$ and $C^i_l$ and $D^j_k$ are the degree $l$ and $k$ part of $C^i$ and $D^j$ respectively. Then $\tot(x) = i+l$ is the total degree of $x$. 
\end{definition}

\begin{proposition}\label{prodcomp}
Let $A$ be a cocommutative gr-Hopf algebra, then the tensor product of graded $A$-complexes is gr-commutative, in the sense that given $\bld{C}$ and $\bld{D}$ graded $A$-complexes, the map $T: \bld{C} \otimes \bld{D} \to \bld{D} \otimes \bld{C}$, given by $x \otimes y \mapsto  (-1)^{\tot(x)\tot(y)} y \otimes x$ is an isomorphism of graded complexes. 
\end{proposition}

\begin{proof}
By \ref{comhom} $T$ is a map of graded $A$-modules. 
We are left to show that it is an isomorphism of graded $A$-complexes. 
We follow the diagram 
$$\xymatrix{ \cdots \ar[r] & (C \otimes D)_{n+1} \ar[r]^{\delta^{\bld{C} \otimes \bld{D}}_{n+1} } \ar[d]^T & (C \otimes D)_n \ar[r]^{\delta^{\bld{C} \otimes \bld{D}}_n}\ar[d]^T & (C \otimes D)_{n-1} \ar[r]^-{\delta^{\bld{C} \otimes \bld{D}}_{n-1}} \ar[d]^T& \cdots \\
\cdots \ar[r] & (D \otimes C)_{n+1} \ar[r]^{\delta^{\bld{D} \otimes \bld{C}}_{n+1} }& (D \otimes C)_n \ar[r]^{\delta^{\bld{D} \otimes \bld{C}}_n} & (D \otimes C)_{n-1} \ar[r]^-{\delta^{\bld{D} \otimes \bld{C}}_{n-1}} & \cdots.}$$

Note that
$$T \circ \delta^{\bld{C} \otimes \bld{D}} (x \otimes y) = (-1)^{(\tot(x)-1)\tot(y)}y \otimes \delta^{\bld{C}}(x) + (-1)^{\tot(x)+\tot(x)(\tot(y)-1)}\delta^{\bld{D}}(y) \otimes x,$$
and 
$$ \delta^{\bld{D} \otimes \bld{C}} \circ T (x \otimes y) = (-1)^{\tot(y)\tot(x)}(\delta^{\bld{D}}(y) \otimes x +(-1)^{\tot(y)}y \otimes \delta^{\bld{C}}(x)).$$
are equal since the signs have the same parity. 
\end{proof}

\begin{theorem}\label{grcomcoh}
Let $A$ be a cocommutative gr-Hopf algebra, then $\HH(A, \kf)$ is a gr-commutative ring with respect the total degree. 
\end{theorem}
\begin{proof}
Let  $$\xymatrix{ \cdots \ar[r] & P_1 \ar[r] & P_0 \ar[r] & \kf \ar[r] & 0}$$
be a graded projective resolution of $\kf$ as an $A$-gr-module and let $\bld{C}$ be the complex 
$$\xymatrix{0 \ar[r] & \uHom_A(P_0, \kf) \ar[r]^{\delta_0} & \uHom_A(P_1, \kf) \ar[r]^{\delta_1} & \uHom_A(P_2, \kf) \ar[r]^{\delta_2} & \cdots}.$$

Then $\bld{C} \otimes \bld{C}$ still corresponds to $\uHom$ applied to a gr-projective resolution of $\kf$.  Moreover, if $x \in C^s_n= \uHom(P_n, \kf)_s$ and $y \in  C^r_m = \uHom(P_m,\kf)_r$ then $x \otimes y \in \uHom(P_n \otimes P_m, \kf)_{s+r} \subset (C \otimes C)_{n+m}^{s+r}$, Hence the tensor product of the complex $\bld{C}$ makes $\HH(A, \kf)$ into a ring.  In fact this product on $\uExt$ is called the \emph{cup product} and in the case of $\HH(A, \kf)$ it coincides with the Yoneda splice (for more details refer to \cite[3.2]{BE1}). We can check that $\delta^{\bld{C} \otimes \bld{C}}(x \otimes y) = \delta^{\bld{C}}(x) \otimes y +(-1)^{\tot(x)} x\otimes \delta^{\bld{C}}(y)$ passes down to a well defined product on $\HH(A, \kf)$.

Let $T: \bld{C} \otimes \bld{C} \to \bld{C} \otimes \bld{ C}$ as in \ref{prodcom},  then by \ref{prodcomp} $T$ is an isomorphism of graded $A$-complexes.  Hence $\HH(A, \kf)$ is a gr-commutative ring with respect the total degree as desired. 
\end{proof}

\begin{rk}
Graded comodules over a gr-commutative Hopf algebra of finite type $B$ correspond to  gr-modules over  its dual $B^\#$ which is a gr-cocommutative Hopf algebra. 

Hence, when working in the category of graded $B$-comodules, $\uExt_B^{*,*}(\kf, \kf)$ is gr-commutative by \ref{grcomcoh}. 
\end{rk}

\bibliography{../resources/biblio}
\bibliographystyle{amsalpha}

\end{document}